\documentclass[reqno,10pt,centertags]{amsart}
\usepackage{amsmath,amsthm,amscd,amssymb,latexsym,esint,upref,stmaryrd,
enumerate,verbatim,yfonts,bm,mathtools,mathrsfs}
\usepackage{color}

\usepackage{hyperref} 
\newcommand*{\mailto}[1]{\href{mailto:#1}{\nolinkurl{#1}}}
\newcommand{\arxiv}[1]{\href{http://arxiv.org/abs/#1}{arXiv:#1}}

\newcommand{\bbC}{{\mathbb{C}}}

\newcommand{\bbN}{{\mathbb{N}}}

\newcommand{\bbR}{{\mathbb{R}}}

\newcommand{\cH}{{\mathcal H}}

\newcommand{\beq}{\begin{equation}}
\newcommand{\enq}{\end{equation}}

\DeclareMathOperator{\supp}{supp}

\renewcommand{\ln}{\text{\rm ln}}

\newcommand{\no}{\notag}
\newcommand{\lb}{\label}
\newcommand{\f}{\frac}

\newcommand{\ol}{\overline}

\newcommand{\wti}{\widetilde}
\newcommand{\Oh}{O}

\newcommand{\bi}{\bibitem}

\renewcommand{\dot}{\overset{\textbf{\Large.}}}

\let\geq\geqslant
\let\leq\leqslant

\makeatletter
\def\theequation{\@arabic\c@equation}

\allowdisplaybreaks 
\numberwithin{equation}{section}

\newtheorem{theorem}{Theorem}[section]

\newtheorem{lemma}[theorem]{Lemma}
\newtheorem{corollary}[theorem]{Corollary}

\theoremstyle{remark}
\newtheorem{remark}[theorem]{Remark}

\newcommand*{\smcirc}{\mathrel{\scalebox{0.4}{$\circ$}}}

\newcommand{\kk}{N}

\begin{document}

\title[Weighted Birman Inequalities with Logarithmic Refinements]{A Sequence of Weighted 
Birman--Hardy--Rellich Inequalities with Logarithmic Refinements} 

\author[F.\ Gesztesy]{Fritz Gesztesy}
\address{Department of Mathematics, 
Baylor University, One Bear Place \#97328,
Waco, TX 76798-7328, USA}
\email{Fritz$\_$Gesztesy@baylor.edu}
\urladdr{http://www.baylor.edu/math/index.php?id=935340}

\author[L.\ L.\ Littlejohn]{Lance L.\ Littlejohn}
\address{Department of Mathematics, 
Baylor University, One Bear Place \#97328,
Waco, TX 76798-7328, USA}
\email{Lance$\_$Littlejohn@baylor.edu}
\urladdr{http://www.baylor.edu/math/index.php?id=53980}

\author[I.\ Michael]{Isaac Michael}
\address{Department of Mathematics, 
Louisiana State University, 
Baton Rouge, LA 70803-4918, USA}
\email{imichael@lsu.edu}

\author[M.\ M.\ H.\ Pang]{Michael M.\ H.\ Pang}
\address{Department of Mathematics,
University of Missouri, Columbia, MO 65211, USA}
\email{pangm@missouri.edu}
\urladdr{https://www.math.missouri.edu/people/pang}


\date{\today}
\subjclass[2010]{Primary: 26D10, 34A40, 35A23; Secondary: 34L10.}
\keywords{Birman--Hardy--Rellich inequalities, logarithmic refinements.}

\begin{abstract} 
The principal aim of this paper is to extend Birman's sequence of integral inequalities 
\begin{align*}
\int_0^\rho dx \, \big| f^{(m )}(x)\big|^{2} \geq \frac{[(2m -1)!!]^{2}}{2^{2m }} \int_0^\rho dx \, 
x^{-2m } |f(x)|^2,& \\
f\in C_{0}^{m }((0,\rho)), \; m  \in \bbN, \quad  \rho \in (0, \infty) \cup \{\infty\},&
\end{align*}
originally obtained in 1961, and containing Hardy's and Rellich's inequality (i.e., $m=1,2$) as special cases, to a sequence of inequalities that incorporates power weights on either side and logarithmic refinements on the right-hand side of the inequality as well.

Introducing iterated logarithms given by
\[
\ln_{1}( \, \cdot \,) = \ln(\, \cdot \,), \quad \ln_{j+1}( \, \cdot \,) = \ln( \ln_{j}(\, \cdot \,)), \quad j \in \bbN,
\]
and  iterated exponentials,
\[
e_{0} = 0, \quad e_{j+1} = e^{e_{j}}, \quad j \in \bbN_{0} = \bbN \cup \{0\}, 
\]
a particular (but representative) extension of Birman's sequence we will prove then reads 
\begin{align*}
&\int_0^{\rho} dx \, x^{\alpha} \big| f^{(m )}(x) \big|^{2}
\geq A(\ell, \alpha) \int_0^{\rho} dx \,  x^{\alpha - 2\ell} \big|f^{(m-\ell)}(x)\big|^{2}   \no \\
&\quad+ B(\ell,\alpha) \sum_{k=1}^{\kk} \int_0^{\rho} dx \, x^{\alpha - 2\ell}\prod_{\ell=1}^{k} 
[\ln_{\ell}(\gamma/x)]^{-2} \big|f^{(m-\ell)}(x)\big|^{2},   \no \\
& \, f \in C_{0}^{\infty}((0, \rho)), \; \ell, m , \kk \in \bbN, \; 1 \leq \ell \leq m, \; 
\alpha \in \bbR, \; \rho, \gamma \in (0,\infty), 
\; \gamma \geq e_{\kk} \rho.
\end{align*}
Here the constants $A(p, \alpha)$ and $B(p, \alpha)$, $p \in \bbN$, are of the form 
\begin{align*}
A(p, \alpha) = \prod_{j=1}^{p} \bigg( \frac{2j - 1 -\alpha}{2} \bigg)^2, \quad 
B(p, \alpha) = \frac{1}{4^{p}} \sum_{k=1}^{p} \ \prod_{\substack{j = 1\\ j \ne k}}^{p} ( 2j - 1 - \alpha )^{2}.  
\end{align*}
The constants $A(\ell, \alpha)$ in the above extension of Birman's inequality are optimal, and so are the conditions on $\gamma$. Moreover, employing a new technique of proof relying on a combination of transforms originally due to Hartman and M\"uller-Pfeiffer, the parameter $\alpha \in \bbR$ in the power weights is now unrestricted, considerably improving on prior results in the literature.

We also indicate a vector-valued version of these inequalities, replacing complex-valued $f(\,\cdot\,)$ by 
$f(\,\cdot\,) \in \cH$, with $\cH$ a complex, separable Hilbert space.     
\end{abstract}

\maketitle


\section{Introduction} \lb{s1}

To be able to describe the content of this paper we start by recalling Birman's infinite sequence of integral inequalities \cite{Bi66}, the sequence of Birman--Hardy--Rellich inequalities of the form 
\begin{align}\lb{1.1}
\begin{split}
\int_a^b dx \, \big| f^{(m )}(x)\big|^{2} \geq \frac{[(2m -1)!!]^{2}}{2^{2m }} \int_a^b dx \,  x^{-2m } |f(x)|^2,& \\
f\in C_{0}^{m }((a,b)), \; m  \in \bbN, \quad 0 \leq a < b \leq \infty,&
\end{split}
\end{align}
which appeared in 1961, and in English translation in 1966 (see also \cite[pp.~83--84]{Gl66}). The case $m =1$ in \eqref{1.1} represents Hardy's celebrated inequality \cite{Ha25}, \cite[Sect.~9.8]{HLP88} (see also 
\cite[Chs.~1, 3, App.]{KMP07}), the case $m=2$ is due to Rellich 
\cite[Sect.~II.7]{Re69} (actually, in the multi-dimensional context). The inequalities \eqref{1.1} and their power weighted generalizations, that is, the first line in \eqref{1.9}, are known to be strict, that is, equality holds in \eqref{1.1}, resp., in the first line in \eqref{1.9} (in fact, for the entire inequality \eqref{1.9}) if and only if $f=0$ on $(a,b)$. Moreover, these inequalities are optimal, meaning, the constants $[(2m-1)!!]^2/2^{2m}$ in \eqref{1.1}, respectively, the constants $A(m,\alpha)$ in \eqref{1.9} are sharp, although, this must be qualified as different authors frequently prove sharpness for different function spaces. In the present one-dimensional context at hand, sharpness of \eqref{1.1} (and typically, it's power weighted version, the first line in \eqref{1.9}), are often proved in an integral form (rather than the currently presented differential form) where $f^{(m)}$ on the left-hand side is replaced by $F$ and $f$ on the right-hand side by $m$ repeated integrals over $F$. For pertinent one-dimensional sources, we refer, for instance, to \cite[p.~3--5]{BEL15}, \cite{CEL99}, \cite[p.~104--105]{Da95}, \cite{GLMW18, GGMR19, Ha25}, \cite[p.~240--243]{HLP88}, \cite[Ch.~3]{KMP07}, \cite[p.~5--11]{KPS17}, \cite{La26, Mu72, PS15}. 
We also note that higher-order Hardy inequalities, including various weight functions, are discussed in \cite[Sect.~5]{Ku85}, \cite[Chs.~2--5]{KMP07}, \cite[Chs.~1--4]{KPS17}, \cite{KW92}, and \cite[Sect.~10]{OK90} (however, Birman's sequence of inequalities \eqref{1.1} is not mentioned in these sources). In addition, there are numerous sources which treat multi-dimensional versions of these inequalities on various domains $\Omega \subseteq \bbR^n$, which, when specialized to radially symmetric functions (e.g., when $\Omega$ represents a ball), imply one-dimensional Birman--Hardy--Rellich-type inequalities with power weights under various restrictions on these weights (cf.\ Remarks \ref{r3.3}\,$(ii)$ and \ref{rA.3}). However, none of the results obtained in this manner imply our principal result, \eqref{1.9}, under optimal hypotheses on $\alpha$ and $\gamma$. We also mention that a large number of these references treat the $L^p$-setting, and in some references $x \in (a,b)$ is replaced by $d(x)$, the distance of $x$ to the boundary of $(a,b)$, respectively, $\Omega$, but this represents quite a different situation (especially in the multi-dimensional context) and hence is not further discussed in this paper. 

The primary aim in this paper is to prove optimal inequalities of the type \eqref{1.1} with additional weights (of power-type on either side of \eqref{1.1}) and logarithmic refinements (i.e., additional, only logarithmically weaker, singularities on the right-hand side of \eqref{1.1}). To describe our new results in detail, we need some preparations and introduce the iterated logarithms $\ln_{j}(\, \cdot \,)$, $j \in \bbN$ (cf.\ \cite{Ha48}, \cite[pp.~324--325]{Ha02})), given by
\begin{align}\lb{1.2}
\ln_{1}( \, \cdot \,) = \ln(\, \cdot \,), \quad \ln_{j+1}( \, \cdot \,) = \ln( \ln_{j}(\, \cdot \,)), \quad j \in \bbN,
\end{align}
and also normalized iterated logarithms $L_{j}(\, \cdot \,)$, $j \in \bbN$ (see, e.g., \cite{BFT03}), 
\begin{align}\lb{1.4}
L_{1}(\, \cdot \,) = \big(1 - \ln(\, \cdot \,)  \big)^{-1}, \quad 
L_{j+1}(\, \cdot \,) = L_{1} (L_{j}(\, \cdot \,)), \quad j \in \bbN.
\end{align}
In addition, we introduce iterated exponentials in the form,
\begin{equation}\lb{1.3}
e_{0} = 0, \quad e_{j+1} = e^{e_{j}}, \quad j \in \bbN_{0} = \bbN \cup \{0\}.
\end{equation}
Moreover, for $m  \in \bbN$ and $\alpha \in \bbR$, we introduce the constants
\begin{align}
&A(m , \alpha) = \prod_{j=1}^{m } \bigg( \frac{2j - 1 -\alpha}{2} \bigg)^2, \lb{1.5a} \\
&B(m , \alpha) = \frac{1}{4^{m }} \sum_{k=1}^{m } \ \prod_{\substack{j = 1\\ j \ne k}}^{m } ( 2j - 1 - \alpha )^{2}. \lb{1.5b}
\end{align}
One observes that 
\begin{align}
B(m , \alpha) &= A(m , \alpha) \sum_{j=1}^{m } \frac{1} {(2j - 1 - \alpha)^{2}}, \quad m  \in \bbN, \; 
\alpha \in \bbR \backslash \{2j - 1\}_{1 \leq j \leq m },   \\
A(m ,0) &= \frac{[(2m -1)!!]^{2}}{2^{2m }},  \quad m  \in \bbN,
\end{align}
in particular, $A(m ,0)$ coincides with the constant in \eqref{1.1}.

The improved Birman inequalities contain additional constants $c_{\ell}(m ,\alpha)$, $\ell = 0,1, \dots, 2m $, which are defined in terms of the polynomial
\begin{align}\lb{1.6}
&P_{m , \alpha}(\lambda) = \sum_{\ell = 0}^{2m } c_{\ell}(m ,\alpha) \lambda^{\ell} = \prod_{j=1}^{m } \bigg( \lambda^2 - \frac{(2j-1-\alpha)^2}{4} \bigg), \quad m \in \bbN, \; \alpha \in \bbR.  
\end{align}

Given the notation introduced in \eqref{1.2}--\eqref{1.6}, we can now describe the principal results proved in this note: Let $\ell, m , \kk \in \bbN$, $1 \leq \ell \leq m$, $\alpha \in \bbR$, $\rho, \gamma \in (0,\infty)$, $\gamma \geq e_{\kk} \rho$, and 
$f \in C_{0}^{\infty}((0, \rho))$. Then the power-weighted Birman--Hardy--Rellich sequence with logarithmic refinements on the interior interval $(0,\rho)$ are of the form 
\begin{align}   
&\int_0^{\rho} dx \, x^{\alpha} \big| f^{(m )}(x) \big|^{2}
\geq A(\ell, \alpha) \int_0^{\rho} dx \,  x^{\alpha - 2\ell} \big|f^{(m-\ell)}(x)\big|^{2}   \no \\
&\quad+ B(\ell,\alpha) \sum_{k=1}^{\kk} \int_0^{\rho} dx \, x^{\alpha - 2\ell}\prod_{\ell=1}^{k} 
[\ln_{\ell}(\gamma/x)]^{-2} \big|f^{(m-\ell)}(x)|^{2}     \no \\
&\quad+ \sum_{j=2}^{\ell} |c_{2j}(\ell,\alpha)| A(j,0) \int_0^{\rho} dx \, x^{\alpha - 2\ell} [\ln(\gamma/x)]^{-2j} 
\big|f^{(m-\ell)}(x)\big|^{2}      \lb{1.9}  \\
&\quad+ \sum_{j=2}^{\ell} |c_{2j}(\ell,\alpha)|  B(j,0) \sum_{k=1}^{\kk - 1} 
\int_0^{\rho} dx \, x^{\alpha - 2\ell} [\ln(\gamma/x)]^{-2j}    \no \\
& \hspace*{5.5cm} \times \prod_{p=1}^{k} [\ln_{p+1}(\gamma/x)]^{-2} \big|f^{(m-\ell)}(x)\big|^{2}.   \no
\end{align}

Moreover, we prove the same sequence of inequalities on the exterior interval $(\rho,\infty)$ for 
$f \in C_{0}^{\infty}((\rho,\infty))$, and finally, both sets of inequalities (exterior and interior) are also proved with the iterated logarithms $\ln_j(\,\cdot\,)$ replaced by the normalized logarithms $L_j(\,\cdot\,)$, $j \in \bbN$. In the latter case an infinite series of logarithmic terms (i.e., the case $N = \infty$ in the analog of \eqref{1.9}) will be permitted. Furthermore, we show that all inequalities are strict, that is, equality holds if and only if $f = 0$ on $(0,\rho)$ (resp., $(\rho, \infty)$). For brevity, a careful comparison of our result with the existing ones in the literature is postponed to Remarks \ref{r3.3} and \ref{rA.3}. The issue of sharpness of constants will be discussed in Appendix \ref{sA}. 

A multi-dimensional version of our approach, focusing on radial and logarithmic refinements of Birman--Hardy--Rellich-type inequalities, will appear in \cite{GLMP20}. 

In Section \ref{s2} we introduce our principal tool, a combined Hartman--M\"uller-Pfeiffer transformation, our main  results are then proved in Section \ref{s3}. In Section \ref{s4} we derive the sequence of power-weighted Birman--Hardy--Rellich inequalities with logarithmic refinements in the vector-valued case, replacing complex-valued $f(\,\cdot\,)$ by $f(\,\cdot\,) \in \cH$, with $\cH$ a complex, separable Hilbert space. Finally, sharpness of the constants $A(m , \alpha)$ is derived in Appendix \ref{sA}.

\section{The Combined Hartman--M\"ueller-Pfeiffer Transformation} \lb{s2}

In this section we introduce an elementary, yet extremely useful, variable transformation, an appropriate combination of special cases of transformations considered by Hartman \cite{Ha48} (see also \cite[p.~324--325]{Ha02}) and 
M\"uller-Pfeiffer \cite[p.~200--207]{MP81}. We now introduce an extension of these transformations by Hartman and M\"uller-Pfeiffer applicable to  power weights and higher-order derivatives. This will be crucial in proving the 
power-weighted Birman--Hardy--Rellich inequalities with logarithmic refinements under most general conditions in our principal Section \ref{s3}. 

Let $m , \kk \in \bbN$ and suppose that 
\begin{equation}\lb{2.1}
\alpha \in \bbR \backslash \{j \,|\, 1 \leq j \leq 2m-1\}.
\end{equation}
Given $f \in C_{0}^{\infty}((e_{\kk}, \infty))$, the transformation
\begin{align}\lb{2.2}
&x = e^{t}, \; x \in (e_N,\infty), \quad  dx = e^{t}dt, \qquad  t \in (e_{\kk-1}, \infty), \\
&f(x) \equiv f(e^t) = e^{[(2m - 1 - \alpha)/2]t}w(t), \quad w \in C_{0}^{\infty}((e_{\kk-1}, \infty)), \lb{2.2a}
\end{align}
yields
\begin{align}\lb{2.3}
\big(x^{\alpha} f^{(m )}(x) \big)^{(m )} = e^{-[(2m + 1 - \alpha)/2]t} \sum_{\ell = 0}^{2m } c_{\ell}(m ,\alpha) w^{(\ell)}(t),
\end{align}
for appropriate constants $c_{\ell}(m ,\alpha)$, $\ell = 0, 1, \dots, 2m $ to be determined next.

The solutions of the differential equation 
\begin{equation}\lb{2.5}
\big(x^{\alpha} f^{(m )}(x) \big)^{(m )}  = 0,
\end{equation}
are linear combinations of the following powers of $x$:
\begin{align}\lb{2.6}
\begin{cases}
x^{j},   &j = 0, 1, \dots, m  - 1, \\
x^{k - \alpha}, &k = m , \dots, 2m  - 1.
\end{cases}
\end{align}
One notes that the solutions \eqref{2.6} are linearly independent due to \eqref{2.1}.

Thus, recalling \eqref{2.2}--\eqref{2.3}, it follows that the solutions of
\begin{equation}\lb{2.7}
\sum_{\ell = 0}^{2m } c_{\ell}(m ,\alpha) w^{(\ell)}(t) = 0, \quad t \in (e_{N-1},\infty),
\end{equation}
are the functions
\begin{align}\lb{2.8}
e^{(\frac{1+\alpha}{2} - m)t}x^{j} = e^{(j+\frac{1+\alpha}{2} - m)t}, \quad j = 0, 1, \dots, m  -1,
\end{align}
and
\begin{align}\lb{2.9}
e^{(\frac{1+\alpha}{2} - m)t}x^{k - \alpha} = e^{(k +\frac{1-\alpha}{2} - m)t}  \quad k = m , \dots, 2m  - 1.
\end{align}
Observe that for $j=0$ and $k = 2m  - 1$,
\begin{align}
\begin{split}
&e^{(j+\frac{1+\alpha}{2} - m)t} = e^{(\frac{1+\alpha}{2} - m)t} \\
&e^{(k +\frac{1-\alpha}{2} - m)t}  = e^{-(\frac{1+\alpha}{2} - m)t}.
\end{split}
\end{align}
For $j=1$ and $k = 2m -2$,
\begin{align}
\begin{split}
&e^{(j+\frac{1+\alpha}{2} - m)t} = e^{(\frac{3+\alpha}{2} - m)t} \\
&e^{(k +\frac{1-\alpha}{2} - m)t}  = e^{-(\frac{3+\alpha}{2} - m)t}.
\end{split}
\end{align}
Continuing iteratively, one concludes that the linearly independent solutions of \eqref{2.7} are of the form
\begin{equation}\lb{2.10}
e^{\pm \frac{1}{2}(2j +1 - 2m  + \alpha)t}, \quad j = 0, 1, \dots, m  - 1,
\end{equation}
By a simple relabeling, given $\alpha \in \bbR \backslash \{j \,|\, 1 \leq j \leq 2m-1\}$, this is equivalent to 
\begin{equation}\lb{2.11}
e^{\pm \frac{1}{2}(2j - 1 - \alpha)t}, \quad j = 1, \dots, m , \; t \in (e_{N-1},\infty), 
\end{equation}
are linearly independent solutions of \eqref{2.7}. 
The zeros of the characteristic polynomial of \eqref{2.7} are thus the constant factors in the exponents of \eqref{2.11}.
Hence, the characteristic polynomial is given by
\begin{align}\lb{2.12}
&P_{m , \alpha}(\lambda) = \sum_{\ell = 0}^{2m } c_{\ell}(m ,\alpha) \lambda^{\ell} \no \\
& \quad = \bigg(\lambda^2 - \frac{(1 - \alpha)^2}{4} \bigg) \bigg(\lambda^2 - \frac{(3 - \alpha)^2}{4} \bigg) \cdots \bigg(\lambda^2 - \frac{(2m -1 - \alpha)^2}{4}  \bigg)  \no \\
& \quad = \prod_{j=1}^{m } \bigg( \lambda^2 - \frac{(2j-1-\alpha)^2}{4} \bigg). 
\end{align}
Thus, the coefficients $c_{\ell}(m ,\alpha)$, $\ell = 0, 1, \dots, 2m $, satisfy the following properties: 
\begin{flalign}\lb{2.4}
(i)& \ c_{2j-1}(m , \alpha) = 0,  \quad  j = 1, \dots, m ;& \no \\[.23cm] 
(ii)& \ c_{2j}(m ,\alpha) = (-1)^{m -j}|c_{2j}(m ,\alpha)|,  \quad  j = 0, 1, \dots, m ;& \no \\[.23cm] 
(iii)& \ |c_{0}(m , \alpha)| = A(m , \alpha); & \\[.23cm] 
(iv)& \ |c_{2}(m ,\alpha)| = 4B(m ,\alpha);& \no  \\[.23cm] 
(v)& \ c_{2m }(m ,\alpha) = 1.& \no
\end{flalign}

Turning our attention to the iterated logarithms, given $\kk \in \bbN$, the transformation \eqref{2.2} 
(i.e., $x = e^{t}$, $x \in (e_N,\infty)$, $t \in (e_{\kk-1}, \infty)$) yields
\begin{align}\lb{2.13}
\sum_{k=1}^{\kk} \prod_{j=1}^{k} [\ln_{j}(x)]^{-2} = t^{-2} + t^{-2}\sum_{k=1}^{\kk - 1} \prod_{j=1}^{k} [\ln_{j}(t)]^{-2}, 
\end{align}
interpreting $\sum_{k=1}^{0}(\, \cdot \, ) =0$.

\section{Power-Weighted Birman--Hardy--Rellich-type Inequalities with Logarithmic Refinements} \lb{s3}

In this section we now establish several improvements of existing power-weighted Birman--Hardy--Rellich  inequalities in the literature by employing the combined Hartman--M\"ueller-Pfeiffer variable transformation from section \ref{s2} in a crucial (and new) manner. These weighted inequalities are proved for both types of iterated logarithms 
$\ln_{j}(\, \cdot \,), \,j \in \bbN$ and $L_{j}(\, \cdot \,), \, j \in \bbN$, and are given on both the exterior interval $(\rho, \infty)$ and interior interval $(0,\rho)$ for any $\rho \in (0,\infty)$.

The principal result of this paper then reads as follows: 

\begin{theorem}\lb{t3.1}
Let $\ell, m , \kk \in \bbN, \alpha \in \bbR,$ and $\rho, \gamma, \tau \in (0,\infty)$. The following hold:\\[1mm]
$(i)$ If $\rho \geq e_{\kk} \gamma$ and $1 \leq \ell \leq m $, then for all $f \in C_{0}^{\infty}((\rho, \infty))$,
\begin{align}\lb{3.33}
&\int_{\rho}^{\infty} dx \, x^{\alpha} \big| f^{(m )}(x) \big|^{2}
\geq A(\ell, \alpha) \int_{\rho}^{\infty} dx \,  x^{\alpha - 2\ell} \big|f^{(m  - \ell)}(x)\big|^{2}   \no \\
&\quad+ B(\ell,\alpha) \sum_{k=1}^{\kk} \int_{\rho}^{\infty}  dx \, x^{\alpha - 2\ell} \prod_{p=1}^{k} [\ln_{p}(x/\gamma)]^{-2} \big|f^{(m  - \ell)}(x)\big|^{2}    \\
&\quad+ \sum_{j=2}^{\ell} |c_{2j}(\ell,\alpha)| A(j,0) \int_{\rho}^{\infty}  dx \, x^{\alpha - 2\ell} [\ln(x/\gamma)]^{-2j} 
\big|f^{(m  - \ell)}(x)\big|^{2}   \no  \\
&\quad + \sum_{j=2}^{\ell} |c_{2j}(\ell,\alpha)| B(j,0) \sum_{k=1}^{\kk - 1} \int_{\rho}^{\infty} dx \, 
x^{\alpha - 2\ell} [\ln(x/\gamma)]^{-2j}\no \\
&\hspace*{5.6cm} \times \prod_{p=1}^{k} [\ln_{p+1}(x/\gamma)]^{-2} \big|f^{(m  - \ell)}(x)\big|^{2}. \no 
\end{align}
$(ii)$ If $\rho \geq \tau$ and $1 \leq \ell \leq m $, then for all $f \in C_{0}^{\infty}((\rho, \infty))$,
\begin{align}\lb{3.34}
&\int_{\rho}^{\infty} dx \, x^{\alpha} \big| f^{(m )}(x) \big|^{2}
\geq A(\ell, \alpha) \int_{\rho}^{\infty} dx \,  x^{\alpha - 2\ell} \big|f^{(m  - \ell)}(x)\big|^{2}  \no \\
&\quad+ B(\ell,\alpha) \sum_{k=1}^{\kk} \int_{\rho}^{\infty} dx \, x^{\alpha - 2\ell} \prod_{p=1}^{k} [L_{p}(\tau/x)]^2 
\big|f^{(m  - \ell)}(x)\big|^{2}  \\
&\quad+ \sum_{j=2}^{\ell} |c_{2j}(\ell,\alpha)| A(j,0) \int_{\rho}^{\infty}  dx \, x^{\alpha - 2\ell} 
[L_{1}(\tau/x)]^{2j} \big|f^{(m  - \ell)}(x)\big|^{2}  \no  \\
&\quad+ \sum_{j=2}^{\ell} |c_{2j}(\ell,\alpha)|  B(j,0) \sum_{k=1}^{\kk -1} \int_{\rho}^{\infty}  dx \, x^{\alpha - 2\ell} 
[L_{1}(\tau/x)]^{2j}  \no \\ 
& \hspace{5.6cm} \times \prod_{p=1}^{k} [L_{p+1}(\tau/x)]^2 \big|f^{(m  - \ell)}(x)\big|^{2}.  \no 
\end{align}
$(iii)$ If $\gamma \geq e_{\kk} \rho$ and $1 \leq \ell \leq m $, then for all $f \in C_{0}^{\infty}((0, \rho))$,
\begin{align}\lb{3.35}
&\int_{0}^{\rho} dx \, x^{\alpha} \big| f^{(m )}(x) \big|^{2}
\geq A(\ell, \alpha) \int_{0}^{\rho} dx \,  x^{\alpha - 2\ell} \big|f^{(m  - \ell)}(x)\big|^{2}   \no \\
&\quad+ B(\ell,\alpha) \sum_{k=1}^{\kk} \int_{0}^{\rho}  dx \, x^{\alpha - 2\ell} 
\prod_{p=1}^{k} [\ln_{p}(\gamma/x)]^{-2} \big|f^{(m  - \ell)}(x)\big|^{2} \\
&\quad+ \sum_{j=2}^{\ell} |c_{2j}(\ell,\alpha)| A(j,0) \int_{0}^{\rho}  dx \, x^{\alpha - 2\ell} 
[\ln(\gamma/x)]^{-2j} \big|f^{(m  - \ell)}(x)\big|^{2}  \no  \\
&\quad+ \sum_{j=2}^{\ell} |c_{2j}(\ell,\alpha)| B(j,0) \sum_{k=1}^{\kk - 1}\int_{0}^{\rho} dx \, 
x^{\alpha - 2\ell} [\ln(\gamma/x)]^{-2j} \no \\
&\hspace*{5.5cm} \times \prod_{p=1}^{k} [\ln_{p+1}(\gamma/x)]^{-2} \big|f^{(m  - \ell)}(x)\big|^{2}. \no 
\end{align}
$(iv)$ If $\tau \geq \rho$ and $1 \leq \ell \leq m $, then for all $f \in  C_{0}^{\infty}((0, \rho))$,
\begin{align}\lb{3.36}
&\int_{0}^{\rho} dx \, x^{\alpha} \big| f^{(m )}(x) \big|^{2}
\geq A(\ell, \alpha)\int_{0}^{\rho}  dx \,  x^{\alpha - 2\ell} \big|f^{(m  - \ell)}(x)\big|^{2} \no \\
&\quad+ B(\ell,\alpha) \sum_{k=1}^{\kk} \int_{0}^{\rho} dx \, x^{\alpha - 2\ell} \prod_{p=1}^{k} 
[L_{p}(x/\tau)]^2 \big|f^{(m  - \ell)}(x)\big|^{2} \\
&\quad+ \sum_{j=2}^{\ell} |c_{2j}(\ell,\alpha)| A(j,0) \int_{0}^{\rho} dx \, x^{\alpha - 2\ell} 
[L_{1}(x/\tau)]^{2j} \big|f^{(m  - \ell)}(x)\big|^{2}  \no  \\
&\quad+ \sum_{j=2}^{\ell} |c_{2j}(\ell,\alpha)|  B(j,0) \sum_{k=1}^{\kk - 1}\int_{0}^{\rho} dx \, x^{\alpha - 2\ell} 
[L_{1}(x/\tau)]^{2j}  \prod_{p=1}^{k} [L_{p+1}(x/\tau)^2 \big|f^{(m  - \ell)}(x)\big|^{2}.  \no 
\end{align}
$(v)$ Inequalities \eqref{3.33}--\eqref{3.36} are strict for $f \not \equiv 0$ on $(\rho,\infty)$, respectively,  
$(0,\rho)$. \\[1mm]
$(vi)$ In the exceptional cases $\alpha \in \{2j - 1\}_{1 \leq j \leq \ell}$ $($i.e., if and only if 
$A(\ell,\alpha) = 0$$)$, the first terms containing $A(\ell,\alpha)$ on the right-hand sides of 
\eqref{3.33}--\eqref{3.36} are to be deleted.
\end{theorem}

We break up the proof of Theorem \ref{t3.1} into four parts. For simplicity, we present the proof in the 
special case $\ell = m$; the general case follows upon replacing $f$ by $f^{(m-\ell)}$ for $\ell = 1, \dots, m $. 

\begin{proof}[Proof of Theorem \ref{t3.1}\,$(i)$] 
Let  $\rho \geq e_{\kk} \gamma$, pick any $f \in C_{0}^{\infty}((\rho, \infty))$, and assume 
that $\alpha \in \bbR$ satisfies \eqref{2.1}. The scaling
\begin{equation}\lb{3.18b}
x = \gamma y, \quad  dx = \gamma dy, \quad g(y) = f(\gamma y), 
\quad y \in (\rho/\gamma, \infty) \subseteq (e_N,\infty), 
\end{equation} 
implies $g \in C_{0}^{\infty}((\rho/\gamma, \infty))$. Applying the transformation \eqref{2.2}, \eqref{2.2a} to $g$, 
that is, employing 
\begin{align}
\begin{split} 
& x/\gamma = y = e^t, \quad dx/\gamma = dy = e^t dt, \quad t \in (\ln(\rho/\gamma),\infty),  \\
& f(x) = g(y) = e^{[(2m -1 - \alpha)/2]t} w(t), \quad w \in C_0^{\infty}((\ln(\rho/\gamma),\infty)), 
\end{split}
\end{align} 
then yields
\begin{align}\lb{3.3}
\big(y^{\alpha} g^{(m )}(y) \big)^{(m )} = e^{-[(2m + 1 - \alpha)/2]t} \sum_{j = 0}^{m } (-1)^{m  - j}|c_{2j}(m ,\alpha)| w^{(2j)}(t),
\end{align}
for $t \in (\ln(\rho/\gamma),\infty)) \subseteq (e_{\kk-1}, \infty)$, and $c_{2j}(m , \alpha)$ as in \eqref{2.4}.
Thus,
\begin{align}\lb{3.4}
(-1)^{m } \big(y^{\alpha} g^{(m )}(y) \big)^{(m )} \ol{g(y)} = e^{-t} \sum_{j = 0}^{m } (-1)^{2m  - j}|c_{2j}(m ,\alpha)| w^{(2j)}(t) \ol{w(t)}.
\end{align}
 Furthermore, \eqref{2.2}, \eqref{2.2a}, and \eqref{2.13} yield
 \begin{align}\lb{3.5}
&y^{\alpha - 2m } |g(y)|^{2} = e^{-t} |w(t)|^{2},   \\
 &y^{\alpha - 2m }\sum_{k=1}^{\kk}\prod_{p=1}^{k}[\ln_{p}(y)]^{-2}|g(y)|^{2}
  = e^{-t} \bigg\{ t^{-2}|w(t)|^{2} + t^{-2}\sum_{k=1}^{\kk - 1}\prod_{p=1}^{k}[\ln_{p}(t)]^{-2}|w(t)|^{2}  \bigg\}, \no 
 \end{align}
 and for $j = 2, \dots, m $,
 \begin{align}\lb{3.5b}
 &y^{\alpha - 2m }[\ln(y)]^{-2j}|g(y)|^{2} = e^{-t} t^{-2j} |w(t)|^{2},  \\
 &y^{\alpha - 2m } [\ln(y)]^{-2j} \sum_{k=1}^{\kk-1}  \prod_{p=1}^{k}[\ln_{p+1}(y)]^{-2}|g(y)|^{2} = e^{-t} t^{-2j}\sum_{k=1}^{\kk-1} \prod_{p=1}^{k}[\ln_{p}(t)]^{-2}|w(t)|^{2}. \no
 \end{align} 
Employing the elementary identity, 
\begin{align}\lb{2.14}
\begin{split}
\int_{a}^{b} dx \, x^{\alpha} \big| f^{(m )}(x) \big|^{2} = (-1)^{m } \int_{a}^{b} dx \, \big(x^{\alpha} f^{(m )}(x) \big)^{(m )}\ol{f(x)},& \\
m  \in \bbN, \; \alpha \in \bbR, \; f \in C_{0}^{\infty}((a,b)), \; 0 \leq a < b \leq \infty,
\end{split} 
\end{align}
and items $(iii)$, $(iv)$ of \eqref{2.4}, it follows from \eqref{3.18b}--\eqref{3.5b} that
\begin{align} 
&\int_{\rho}^{\infty} dx \, \bigg\{ x^{\alpha} \big| f^{(m )}(x) \big|^{2} - A(m ,\alpha) x^{\alpha - 2m } |f(x)|^{2} \no \\
&\hspace*{1.4cm} - B(m ,\alpha) x^{\alpha - 2m }\sum_{k=1}^{\kk}\prod_{p=1}^{k}[\ln_{p}(x/\gamma)]^{-2}|f(x)|^{2} \no \\
&\hspace*{1.4cm} - \sum_{j=2}^{m } |c_{2j}(m ,\alpha)| A(j,0) x^{\alpha - 2m } [\ln(x/\gamma)]^{-2j} |f(x)|^{2}  
\no  \\
&\hspace*{1.4cm} -  \sum_{j=2}^{m } |c_{2j}(m ,\alpha)|  B(j,0)  x^{\alpha - 2m } [\ln(x/\gamma)]^{-2j} \sum_{k=1}^{\kk - 1}  \prod_{p=1}^{k} [\ln_{p+1}(x/\gamma)]^{-2}|f(x)|^{2} \bigg\} \no \\
&\quad = \gamma^{\alpha - 2m  + 1} \int_{\rho/\gamma}^{\infty} dy \, \bigg\{ y^{\alpha} \big| g^{(m )}(y) \big|^{2} - A(m ,\alpha) y^{\alpha - 2m } |g(y)|^{2} \no \\
&\hspace*{3.5cm} - B(m ,\alpha) y^{\alpha - 2m }\sum_{k=1}^{\kk}\prod_{p=1}^{k}[\ln_{p}(y)]^{-2}|g(y)|^{2} \no \\
&\hspace*{3.5cm} - \sum_{j=2}^{m } |c_{2j}(m ,\alpha)| A(j,0) y^{\alpha - 2m } [\ln(y)]^{-2j} |g(y)|^{2}  \no  \\
&\hspace*{2.2cm} -  \sum_{j=2}^{m } |c_{2j}(m ,\alpha)|  B(j,0)  y^{\alpha - 2m } [\ln(y)]^{-2j} \sum_{k=1}^{\kk - 1}  \prod_{p=1}^{k} [\ln_{p+1}(y)]^{-2}|g(y)|^{2} \bigg\} \no \\
&\quad = \gamma^{\alpha - 2m  + 1}  
\bigg \{ \sum_{j=0}^{m }|c_{2j}(m ,\alpha)| \int_{\ln(\rho/\gamma)}^{\infty} dt \, \big|w^{(j)}(t) \big|^{2} - A(m ,\alpha) \int_{\ln(\rho/\gamma)}^{\infty} dt \,  |w(t)|^{2} \no \\
&\hspace*{2.4cm} - B(m ,\alpha)\int_{\ln(\rho/\gamma)}^{\infty} dt \, t^{-2}|w(t)|^{2}    \no \\ 
&\hspace*{2.4cm} - B(m ,\alpha)\sum_{k=1}^{\kk - 1} \int_{\ln(\rho/\gamma)}^{\infty} dt \, t^{-2}\prod_{p=1}^{k}[\ln_{p}(t)]^{-2}|w(t)|^{2} \no \\
&\hspace*{2.4cm} 
- \sum_{j=2}^{m } |c_{2j}(m ,\alpha)| A(j,0) \int_{\ln(\rho/\gamma)}^{\infty} dt \, t^{-2j} |w(t)|^{2}  \no \\
&\hspace*{2.4cm} -  \sum_{j=2}^{m } |c_{2j}(m ,\alpha)|   B(j,0) \sum_{k=1}^{\kk - 1} 
\int_{\ln(\rho/\gamma)}^{\infty} dt \, t^{-2j} \prod_{p=1}^{k} [\ln_{p}(t)]^{-2}|w(t)|^{2} \bigg \}  \no \\
& \quad =  \gamma^{\alpha - 2m  + 1}  
\bigg \{  \sum_{j=1}^{m }|c_{2j}(m ,\alpha)| \int_{\ln(\rho/\gamma)}^{\infty} dt \, \big|w^{(j)}(t) \big|^{2} \no \\
&\hspace*{2.4cm} 
- \sum_{j=1}^{m } |c_{2j}(m ,\alpha)| A(j,0) \int_{\ln(\rho/\gamma)}^{\infty} dt \, t^{-2j} |w(t)|^{2}  \no \\
&\hspace*{2.4cm} -  \sum_{j=1}^{m } |c_{2j}(m ,\alpha)|   
B(j,0) \sum_{k=1}^{\kk - 1} \int_{\ln(\rho/\gamma)}^{\infty} dt \, t^{-2j} 
\prod_{p=1}^{k} [\ln_{p}(t)]^{-2}|w(t)|^{2} \bigg\} \no \\
& \quad = \gamma^{\alpha - 2m  + 1} \sum_{j=1}^{m }|c_{2j}(m ,\alpha)| 
\bigg\{ \int_{\ln(\rho/\gamma)}^{\infty} dt \, \big|w^{(j)}(t) \big|^{2} 
- A(j,0) \int_{\ln(\rho/\gamma)}^{\infty} dt \, t^{-2j} |w(t)|^{2}  \no \\
&\hspace*{4.6cm} - B(j,0) \sum_{k=1}^{\kk - 1} 
\int_{\ln(\rho/\gamma)}^{\infty} dt \, t^{-2j} \prod_{p=1}^{k} [\ln_{p}(t)]^{-2}|w(t)|^{2} \bigg\},   \no \\
& \hspace*{7.5cm} w \in C_0^{\infty}((\ln(\rho/\gamma),\infty)),   \lb{3.6} 
\end{align}
interpreting $\sum_{k=1}^{0}(\, \cdot \,) = 0$.

Hence, part $(i)$, for $\alpha \in \bbR \backslash \{j \,|\, 1 \leq j \leq 2m-1\}$, follows via induction over $\kk \in \bbN$. Indeed, for $\kk = 1$ equality \eqref{3.6} yields (cf.\ \eqref{2.2a}) 
\begin{align}\lb{3.6b}
&\int_{\rho}^{\infty} dx \, \bigg\{ x^{\alpha} \big| f^{(m )}(x) \big|^{2} - A(m ,\alpha) x^{\alpha - 2m } |f(x)|^{2}-  B(m ,\alpha) x^{\alpha - 2m }[\ln(x/\gamma)]^{-2}|f(x)|^{2} \no \\
&\qquad \qquad - \sum_{j=2}^{m } |c_{2j}(m ,\alpha)| A(j,0) x^{\alpha - 2m } [\ln(x/\gamma)]^{-2j} |f(x)|^{2}  \bigg\} \no \\
& \quad = \gamma^{\alpha - 2m  + 1} \sum_{j=1}^{m }|c_{2j}(m ,\alpha)| 
\bigg\{ \int_{\ln(\rho/\gamma)}^{\infty} dt \, \big|w^{(j)}(t) \big|^{2} - A(j,0) \int_{\ln(\rho/\gamma)}^{\infty} dt \, t^{-2j} |w(t)|^{2}  \bigg\} \no  \\
& \quad \geq 0, \quad w \in C_0^{\infty}((\ln(\rho/\gamma),\infty)), 
\end{align}
by \eqref{1.1} as a sum of unweighted Birman--Hardy--Rellich-type inequalities. 
Assuming \eqref{3.33} holds for $\kk-1 \in \bbN$ then reapplying \eqref{3.6} proves \eqref{3.33} for $\kk \in \bbN$. Strictness also follows by induction over $\kk \in \bbN$ since $f \not \equiv 0$ implies $w \not \equiv 0$ by \eqref{2.2}, \eqref{2.2a} so that \eqref{3.6b}, and by induction, \eqref{3.6} is strictly positive. 

The case $\alpha \in \{j \,|\, 1 \leq j \leq 2m-1\}$ then follows by taking the limits $\alpha \to k \in\{j \,|\, 1 \leq j \leq 2m-1\}$, noting that $A(m , \alpha)$, $B(m ,\alpha)$, and $c_{2j}(m ,\alpha)$ are continuous as polynomials in $\alpha \in \bbR$. This completes the proof of part $(i)$.
\end{proof} 

\begin{proof}[Proof of Theorem \ref{t3.1}\,$(ii)$]
By taking limits as in part $(i)$, it suffices once more to consider 
$\alpha \in \bbR \backslash  \{2j-1\}_{1 \leq j \leq m}$. 
Let $\rho \geq \tau$ and pick any $f \in C_{0}^{\infty}((\rho,\infty))$. The scaling
\begin{equation}\lb{3.18c}
x = \tau y, \quad  dx = \tau dy, \quad  g(y) = f(\tau y), \quad y \in (\rho/\tau, \infty), 
\end{equation} 
yields $g \in C_{0}^{\infty}((\rho/\tau, \infty)) \subseteq  C_{0}^{\infty}((1, \infty))$.
One modifies the transformation \eqref{2.2}, \eqref{2.2a} applied to $g$ by 
\begin{align}\lb{3.7}
\begin{split}
&y = e^{t - 1}, \quad  dy = e^{t-1}dt, \quad  t \in (1, \infty), \\
&g(y) \equiv g(e^{t-1}) = e^{[(2m - 1 - \alpha)/2](t-1)}v(t), \quad v \in C_{0}^{\infty}((1,\infty)), 
\end{split}
\end{align}
where $v$ is given by
\begin{equation}
v(t) := w(t-1), \quad t \in (1, \infty),
\end{equation}
with $w \in C_{0}^{\infty}((0,\infty))$.
Setting
\begin{equation}
s = t - 1, \quad ds = dt,
\end{equation}
 and noting
\begin{equation}
\frac{d}{dt}v(t) = \frac{d}{ds}w(s),
\end{equation}
yields, similarly to \eqref{3.3},
\begin{align}\lb{3.8}
\begin{split}
&\big(y^{\alpha} g^{(m )}(y) \big)^{(m )} 
= e^{-[(2m + 1 - \alpha)/2]s} \sum_{\ell = 0}^{2m } c_{\ell}(m ,\alpha) w^{(\ell)}(s) \\
&\quad = e^{-[(2m + 1 - \alpha)/2](t-1)} \sum_{\ell = 0}^{2m } c_{\ell}(m ,\alpha) v^{(\ell)}(t). 
\end{split}
\end{align}
Hence, an analogous argument as in section \ref{s2} shows the constants $c_{\ell}(m ,\alpha)$ satisfy $(i)$--$(v)$ in \eqref{2.4} as before.
Therefore by \eqref{3.8},
\begin{align}\lb{3.9}
&(-1)^{m } \big(y^{\alpha} g^{(m )}(y) \big)^{(m )} \ol{g(y)} =e^{1-t}\sum_{j = 0}^{m } (-1)^{2m  - j}|c_{2j}(m ,\alpha)| v^{(2j)}(t) \ol{v(t)}.
\end{align}
Now, \eqref{3.7} yields
\begin{align}
L_{1}(1/y) &= \big(1 - \ln(1/y) \big)^{-1} = \big(1 - \ln(e^{1-t}) \big)^{-1} = t^{-1},
\end{align}
and
\begin{align}
L_{2}(1/y) \! &= \! L_{1}(L_{1}(1/y)) \! = \! L_{1}(1/t).
\end{align}
Inductively, we see that
\begin{align}\lb{3.10}
L_{1}(1/y)  = t^{-1}, \quad L_{j}(1/y) = L_{j-1}(1/t), \quad j=2,3,\dots
\end{align}
Hence, 
 \begin{align}\lb{3.11a}
 &y^{\alpha - 2m } |g(y)|^{2} = e^{1-t} |v(t)|^{2},   \\
 &y^{\alpha - 2m }\sum_{k=1}^{\kk}\prod_{p=1}^{k}L_{p}^{2}(1/y)|g(y)|^{2}
 = e^{1-t} \bigg\{ t^{-2}|v(t)|^{2} + t^{-2}\sum_{k=1}^{\kk - 1}\prod_{p=1}^{k}L_{p}^{2}(1/t)|v(t)|^{2}  \bigg\}, \no 
 \end{align}
 and for $j = 2, \dots, m $,
 \begin{align}\lb{3.11b}
 &y^{\alpha - 2m }L_{1}^{2j}(1/y)|g(y)|^{2} = e^{1-t} t^{-2j} |v(t)|^{2},  \\
 &y^{\alpha - 2m } L_{1}^{2j}(1/y) \sum_{k=1}^{\kk-1}  \prod_{p=1}^{k}L_{p+1}^{2}(1/y)|g(y)|^{2} = e^{1-t} t^{-2j}\sum_{k=1}^{\kk-1} \prod_{p=1}^{k}L_{p}^{2}(1/t)|v(t)|^{2}. \no
 \end{align} 
 Again recalling \eqref{2.14} and $(iii)$--$(iv)$ of \eqref{2.4}, \eqref{3.9}, \eqref{3.11a}, and \eqref{3.11b} yield 
 (cf.\ \eqref{3.7})
\begin{align}
&\int_{\rho}^{\infty} dx \, \bigg\{ x^{\alpha} \big| f^{(m )}(x) \big|^{2} - A(m ,\alpha) x^{\alpha - 2m } |f(x)|^{2} \no \\
&\hspace*{1.4cm} - B(m ,\alpha) x^{\alpha - 2m }\sum_{k=1}^{\kk}\prod_{p=1}^{k}L_{p}^{2}(\tau/x)|f(x)|^{2} \no \\
&\hspace*{1.4cm} - \sum_{j=2}^{m } |c_{2j}(m ,\alpha)| A(j,0) x^{\alpha - 2m } L_{1}^{2j}(\tau/x)  |f(x)|^{2}  \no  \\
&\hspace*{1.4cm} - \sum_{j=2}^{m } |c_{2j}(m ,\alpha)|  B(j,0)  x^{\alpha - 2m } L_{1}^{2j}(\tau/x)\sum_{k=1}^{\kk-1}  \prod_{p=1}^{k}L_{p+1}^{2}(\tau/x)|f(x)|^{2} \bigg\} \no \\
& \quad = \tau^{\alpha-2m +1}\sum_{j=1}^{m }|c_{2j}(m ,\alpha)| \bigg\{\int_{1}^{\infty} dt \, \big|v^{(j)}(t) \big|^{2} - A(j,0) \int_{1}^{\infty}dt \, t^{-2j} |v(t)|^{2}    \no \\
&\hspace*{4.6cm} - B(j,0) \sum_{k=1}^{\kk - 1} \int_{1}^{\infty} dt \, t^{-2j} \prod_{p=1}^{k} L_{p}^{2}(1/t)|v(t)|^{2} \bigg\}, \no \\
& \hspace*{8.4cm}  v \in C_0^{\infty}((1,\infty)),    \lb{3.26a} 
\end{align}
and the proof again follows by induction over $\kk \in \bbN$.
\end{proof}

\begin{proof}[Proof of Theorem \ref{t3.1}\,$(iii)$]
Consider again $\alpha \in \bbR \backslash \{j \,|\, 1 \leq j \leq 2m-1\}$. Let $\gamma \geq e_{\kk} \rho$ 
and pick any $f \in C_{0}^{\infty}((0,\rho))$. The scaling
\begin{equation}\lb{3.27}
x = \gamma y, \quad    dx = \gamma dy, \quad  y \in (0, \rho/\gamma), \quad g(y) = f(\gamma y),
\end{equation} 
yields $g \in C_{0}^{\infty}((0, \rho/\gamma))$. Slightly modifying the transformation \eqref{2.2}, \eqref{2.2a} 
applied to $g$ leads to 
\begin{align}\lb{3.21}
\begin{split}
&y = e^{-t}, \quad  dy = -e^{-t}dt, \quad  t \in (\ln(\gamma/ \rho), \infty), \\
&g(y) \equiv g(e^{-t}) = e^{-[(2m - 1 - \alpha)/2]t}u(t), \quad u \in C_{0}^{\infty}((\ln(\gamma/ \rho), \infty)).
\end{split}
\end{align}
This implies 
\begin{align}\lb{3.22}
&\big(y^{\alpha} g^{(m )}(y) \big)^{(m )} 
 = e^{[(2m + 1 - \alpha)/2]t} \sum_{\ell = 0}^{2m } (-1)^{\ell} {c_{\ell}}(m ,\alpha) u^{(\ell)}(t),  
\end{align}
and hence, $(i)$--$(v)$ in \eqref{2.4} still hold. Thus, 
\begin{align}\lb{3.23}
&(-1)^{m } \big(y^{\alpha} g^{(m )}(y) \big)^{(m )} \ol{g(y)} =e^{t}\sum_{j = 0}^{m } (-1)^{2m  - j}|c_{2j}(m ,\alpha)|u^{(2j)}(t) \ol{u(t)}.
\end{align}
 Furthermore, 
 \begin{align}
&y^{\alpha - 2m } |g(y)|^{2} = e^{t} |u(t)|^{2},      \lb{3.24} \\
&y^{\alpha - 2m }\sum_{k=1}^{\kk}\prod_{p=1}^{k}[\ln_{p}(1/y)]^{-2}|g(y)|^{2} = e^{t} \bigg\{ t^{-2}|u(t)|^{2} + t^{-2}\sum_{k=1}^{\kk - 1}\prod_{p=1}^{k}[\ln_{p}(t)]^{-2}|u(t)|^{2}  \bigg\}, \no 
 \end{align}
 and for $j = 2, \dots, m $,
 \begin{align}\lb{3.25}
 &y^{\alpha - 2m }[\ln(1/y)]^{-2j}|g(y)|^{2} = e^{t} t^{-2j} |u(t)|^{2},  \\
 &y^{\alpha - 2m } [\ln(1/y)]^{-2j} \sum_{k=1}^{\kk-1}  \prod_{p=1}^{k}[\ln_{p+1}(1/y)]^{-2}|g(y)|^{2} = e^{t} t^{-2j}\sum_{k=1}^{\kk-1} \prod_{p=1}^{k}[\ln_{p}(t)]^{-2}|u(t)|^{2}. \no
 \end{align} 
 Applying \eqref{3.23}--\eqref{3.25} yields
\begin{align} 
&\int_{0}^{\rho} dx \, \bigg\{ x^{\alpha} \big| f^{(m )}(x) \big|^{2} - A(m ,\alpha) x^{\alpha - 2m } |f(x)|^{2} \no \\
&\hspace*{1.3cm} - B(m ,\alpha) x^{\alpha - 2m }\sum_{k=1}^{\kk}\prod_{p=1}^{k}[\ln_{p}(\gamma/x)]^{-2}|f(x)|^{2} \no \\
&\hspace*{1.3cm} - \sum_{j=2}^{m } |c_{2j}(m ,\alpha)| A(j,0) x^{\alpha - 2m } [\ln(\gamma/x)]^{-2j} |f(x)|^{2}   
\no \\
&\hspace*{1.3cm} -  \sum_{j=2}^{m } |c_{2j}(m ,\alpha)|  B(j,0)  x^{\alpha - 2m } [\ln(\gamma/x)]^{-2j} \sum_{k=1}^{\kk - 1}  \prod_{p=1}^{k} [\ln_{p+1}(\gamma/x)]^{-2}|f(x)|^{2} \bigg\} \no \\
&\quad = \gamma^{\alpha-2m +1} \sum_{j=1}^{m }|c_{2j}(m ,\alpha)| 
\bigg\{ \int_{\ln(\gamma/ \rho)}^{\infty} dt \, \big|u^{(j)}(t) \big|^{2} - A(j,0)  
\int_{\ln(\gamma/ \rho)}^{\infty} dt \, t^{-2j} |u(t)|^{2}  \no \\
&\hspace*{4.6cm} - B(j,0) \sum_{k=1}^{\kk - 1} 
\int_{\ln(\gamma/ \rho)}^{\infty}dt \, t^{-2j} \prod_{p=1}^{k} [\ln_{p}(t)]^{-2}|u(t)|^{2} \bigg\},   \no \\
& \hspace*{7.5cm}  u \in C_0^{\infty}((\ln(\gamma/ \rho),\infty)),     \lb{3.26}
\end{align}
and the proof follows by induction over $\kk \in \bbN$, as before.
\end{proof}

\begin{proof}[Proof of Theorem \ref{t3.1}\,$(iv)$]
Once more, consider $\alpha \in \bbR \backslash \{j \,|\, 1 \leq j \leq 2m-1\}$. 
Suppose $\tau \geq \rho$, $f \in C_{0}^{\infty}((0,\rho))$, and use the scaling
\begin{equation}\lb{3.32}
x = \tau y, \quad    dx = \tau dy, \quad  y \in (0, \rho/\tau), \quad g(y) = f(\tau y),
\end{equation} 
so that $g \in C_{0}^{\infty}((0, \rho/\tau)) \subseteq  C_{0}^{\infty}((0, 1))$. Next, one  
applies the modified transformation
\begin{align}\lb{3.28}
\begin{split}
&y = e^{-t + 1}, \quad  dy = -e^{-t+1}dt, \quad  t \in (1, \infty), \\
&g(y) \equiv g(e^{-t+1}) = e^{[(2m - 1 - \alpha)/2](1-t)}v(t), \quad v \in C_{0}^{\infty}((1,\infty)),
\end{split}
\end{align}
where $v$ is given by
\begin{equation}
v(t) := w(1-t), \quad t \in (1, \infty),
\end{equation}
with $w \in C_{0}^{\infty}((-\infty,0))$. 
Therefore
\begin{align}\lb{3.29}
&(-1)^{m } \big(y^{\alpha} g^{(m )}(y) \big)^{(m )} \ol{g(y)} =e^{t-1}\sum_{j = 0}^{m } (-1)^{2m  - j}|c_{2j}(m ,\alpha)| v^{(2j)}(t) \ol{v(t)}.
\end{align}
Also,
 \begin{align}\lb{3.30}
 &y^{\alpha - 2m } |g(y)|^{2} = e^{t-1} |v(t)|^{2},   \\
 &y^{\alpha - 2m }\sum_{k=1}^{\kk}\prod_{p=1}^{k}L_{p}^{2}(y)|g(y)|^{2}
  = e^{t-1} \bigg\{ t^{-2}|v(t)|^{2} + t^{-2}\sum_{k=1}^{\kk - 1}\prod_{p=1}^{k}L_{p}^{2}(1/t)|v(t)|^{2}  \bigg\}, \no 
 \end{align}
 and for $j = 2, \dots, m $,
 \begin{align}\lb{3.31}
 &y^{\alpha - 2m }L_{1}^{2j}(y)|g(y)|^{2} = e^{t-1} t^{-2j} |v(t)|^{2},  \\
 &y^{\alpha - 2m } L_{1}^{2j}(y) \sum_{k=1}^{\kk-1}  \prod_{p=1}^{k}L_{p+1}^{2}(y)|g(y)|^{2} = e^{t-1} t^{-2j}\sum_{k=1}^{\kk-1} \prod_{p=1}^{k}L_{p}^{2}(1/t)|v(t)|^{2}. \no
 \end{align} 
Hence,
\begin{align}
&\int_{0}^{\rho} dx \, \bigg\{ x^{\alpha} \big| f^{(m )}(x) \big|^{2} - A(m ,\alpha) x^{\alpha - 2m } |f(x)|^{2} \no \\
&\hspace*{1.3cm} - B(m ,\alpha) x^{\alpha - 2m }\sum_{k=1}^{\kk}\prod_{p=1}^{k}L_{p}^{2}(x/\tau)|f(x)|^{2} \no \\
&\hspace*{1.3cm} - \sum_{j=2}^{m } |c_{2j}(m ,\alpha)| A(j,0) x^{\alpha - 2m } L_{1}^{2j}(x/\tau)  |f(x)|^{2}  \no  \\
&\hspace*{1.3cm} - \sum_{j=2}^{m } |c_{2j}(m ,\alpha)|  B(j,0)  x^{\alpha - 2m } L_{1}^{2j}(x/\tau)\sum_{k=1}^{\kk-1}  \prod_{p=1}^{k}L_{p+1}^{2}(x/\tau)|f(x)|^{2} \bigg\} \no \\
& \quad = \tau^{\alpha-2m +1}\sum_{j=1}^{m }|c_{2j}(m ,\alpha)| \bigg\{\int_{1}^{\infty} dt \, \big|v^{(j)}(t) \big|^{2} - A(j,0) \int_{1}^{\infty} dt \, t^{-2j} |v(t)|^{2} \no \\
&\hspace*{4.6cm} - B(j,0) \sum_{k=1}^{\kk - 1} \int_{1}^{\infty}dt \, t^{-2j} 
\prod_{p=1}^{k} L_{p}^{2}(1/t)|v(t)|^{2} \bigg\},    \no \\
& \hspace*{8.4cm}  v \in C_0^{\infty}((1,\infty)), 
\end{align}
and the proof follows again by induction over $\kk \in \bbN$. 
\end{proof}

Theorem \ref{t3.1}\,$(ii)$,\,$(iv)$ can be further improved by replacing the $\kk$-th sum with an infinite series. See, for example, \cite{BFT03, GM13, TZ07} for similar results and discussions of the convergence of the series $\sum_{k=1}^{\infty} \prod_{j=1}^{k} L_{j}^{2}(s)$ for $s \in (0,1)$.

\begin{corollary}\lb{c3.2}
Let $\ell, m  \in \bbN, \alpha \in \bbR,$ and $\rho, \tau \in (0,\infty)$. Then \eqref{3.34} and \eqref{3.36} 
extend to $N = \infty$. 
\end{corollary}
\begin{proof}
It suffices to discuss the proof of \eqref{3.34}. Given $f \in C_{0}^{\infty}((\rho, \infty))$, 
Theorem \ref{t3.1}\,$(ii)$ implies that \eqref{3.34} holds for any $\kk \in \bbN$.
Thus, by taking $\kk \uparrow \infty$ and recalling that increasing sequences bounded above are convergent, \eqref{3.34} holds with $N = \infty$. 
\end{proof}

To put our results in perspective and to compare with existing results in the literature, we offer some  comments next.

\begin{remark}\lb{r3.3}
$(i)$ Theorem \ref{t3.1}\,$(i)$,$(ii)$ (resp., Theorem \ref{t3.1}\,$(iii)$,$(iv)$) extends to $N = \rho = 0$ (resp., 
$N=0$, $\rho = \infty$) upon disregarding all logarithmic terms (i.e., upon putting 
$B(\ell,\alpha) = c_{2j}(\ell,\alpha)= 0$, $2 \leq j \leq \ell$, $1 \leq \ell \leq m$), we omit the details. \\[1mm] 
$(ii)$ Originally, logarithmic refinements of Hardy's inequality started with oscillation theoretic considerations going back to Hartman \cite{Ha48} (see also \cite[p.~324--325]{Ha02}) and have been used in connection with Hardy's inequality in \cite{Ge84, GU98}, and more recently in \cite{GL18, GLMP19}. Since then there has been enormous activity in this context and we mention, for instance,  \cite{ACR01, AE05, AFT09, AGS06, AS02, AS09, AS06, Al76, AVV10, AH12, AGG06, Av17}, 
\cite[Chs.~3, 5]{BEL15}, \cite{BFT03, BFT18, Be89, CM12, Co10, DH98, DHA04, DHA04b, DHA05, DHA12, Du09, Du10, Du14, DLT19, DLT20, FT02, GGM03, GL18, GP18, GM08, GM11}, \cite[Chs.~2,6,7]{GM13}, \cite{Hi04, II15, MOW13, MOW15, MOW16, MOW17, Mi00, Mo12, Mu14, NN20, NY75}, \cite[Sect.~2.7]{Re69}, \cite{RS17, RY20, Sa19, So94, Ta15, TZ07, Ya99}. The vast majority of these references deals with analogous multi-dimensional settings (relevant to our setting in particular in the case of radially symmetric functions), most also in the $L^p$-context. \\[1mm] 
$(iii)$ For  $m  \geq 2$ these inequalities are new in the following sense: The weight parameter 
$\alpha \in \bbR$ is now unrestricted (as opposed to prior results, see item $(ii)$ of this remark) and at 
the same time the conditions on the logarithmic parameters $\gamma$ and $\tau$ are sharp. Moreover, the two integral terms containing $c_{2j}(m ,\alpha)$ are new in this generality (we note that a single term of the type $x^{-2m} [\ln(\gamma/x)]^{-4}$ appeared in \cite{AGS06} and \cite{CM12}; and \cite[Ch.~6]{MP81} discusses 
sums involving even powers of $[\ln(\gamma/x)]^{-1}$). We also note that the inequalities are proved for both 
iterated logarithms $\ln_{j}(\, \cdot \,)$ and $L_{j}(\, \cdot \,)$, $j \in \bbN$, and finally they are proved on both the exterior interval $(\rho, \infty)$ and interior interval $(0,\rho)$ for any $\rho \in (0,\infty)$.
\hfill $\diamond$
\end{remark}

We conclude this section by extending Theorem \ref{t3.1} from $C_0^{\infty}$-functions to functions in 
appropriately weighted Sobolev spaces.

To this end we introduce the norms on $C_0^{\infty}((a,b))$,  
\begin{align}
\begin{split}
\|f\|_{m,\alpha}^2 = \sum_{k=0}^m \int_a^b dx \, x^{\alpha} \big|f^{(k)}(x)\big|^2, \quad 
 |||f|||_{m,\alpha}^2 = \int_a^b dx \, x^{\alpha} \big|f^{(m)}(x)\big|^2,& \\   
0 \leq a < b \leq \infty, \; m \in \bbN, \; \alpha \in \bbR, \;  f \in C_0^{\infty}((a,b)).& 
\end{split} 
\end{align}
and define the weighted Sobolev spaces
\begin{equation}
H_0^m\big((a,b); x^{\alpha} dx\big) = \ol{C_0^{\infty}((a,b))}^{\| \, \cdot \,\|_{m,\alpha}}, \quad m \in \bbN, 
\; \alpha \in \bbR,     \lb{3.42} 
\end{equation}
and the corresponding homogeneous weighted Sobolev spaces 
\begin{equation}
\dot {H_0^m}\big((a,b); x^{\alpha} dx\big) = \ol{C_0^{\infty}((a,b))}^{||| \, \cdot \, |||_{m,\alpha}}, 
\quad m \in \bbN, \; \alpha \in \bbR.    \lb{3.43} 
\end{equation}
In the case $b < \infty$ we also note the following (higher-order) weighted Poincar\'e-type inequality: 

\begin{lemma} \lb{l3.4}
Let $\rho \in (0, \infty)$, $k,m \in \bbN$, $0 \leq k \leq m-1$, $\alpha \in \bbR$. Then there exists 
$C_{k,m} = C(k,m,\alpha,\rho) \in (0,\infty)$ such that 
\begin{align}
& C_{k,m} \big\|f^{(k)}\big\|_{L^2((0,\rho); x^{\alpha}dx)}^2 = C_{k,m} |||f|||_{k,\alpha}^2 
\leq |||f|||_{m,\alpha}^2 = \big\|f^{(m)}\big\|_{L^2((0,\rho); x^{\alpha}dx)}^2,  \no \\  
& \hspace*{8.5cm} f \in C_0^{\infty}((0,\rho)).     \lb{3.48} 
\end{align}
\end{lemma}
\begin{proof}
If $A(m-k,\alpha) \neq 0$, one can use the simplest inequality in Theorem \ref{t3.1}\,$(iii)$ to conclude for 
$f \in C_0^{\infty}((0,\rho))$, 
\begin{align}
\begin{split} 
\int_0^{\rho} dx \, x^{\alpha} \big|f^{(m)}(x)\big|^2 &\geq A(m-k,\alpha) \int_0^{\rho} dx \, x^{\alpha - 2(m-k)} 
\big|f^{(k)}(x)\big|^2   \\
& \geq A(m-k,\alpha) \rho^{-2(m-k)} \int_0^{\rho} dx \, x^{\alpha} \big|f^{(k)}(x)\big|^2. 
\end{split}
\end{align}  
If $A(m-k,\alpha) = 0$, one uses the next simplest inequality  in Theorem \ref{t3.1}\,$(iii)$ to infer 
\begin{align}
& \int_0^{\rho} dx \, x^{\alpha} \big|f^{(m)}(x)\big|^2 \geq B(m-k,\alpha) \int_0^{\rho} dx \, x^{\alpha - 2(m-k)} 
[\ln(\gamma/x)]^{-2} \big|f^{(k)}(x)\big|^2     \no \\
& \quad \geq B(m-k,\alpha) \bigg(\int_0^{\eta} + \int_{\eta}^{\rho}\bigg) dx \, x^{\alpha - 2(m-k)} 
[\ln(\gamma/x)]^{-2} \big|f^{(k)}(x)\big|^2     \no \\
& \quad \geq C_{k,m}  \int_0^{\rho} dx \, x^{\alpha} \big|f^{(k)}(x)\big|^2,
\end{align}  
where $\eta \in (0,\rho)$ is chosen such that $x^{- 2(m-k)} [\ln(\gamma/x)]^{-2}$ is strictly monotonically decreasing 
on the interval $(0,\eta)$. 
\end{proof}

Thus, for $\rho \in (0, \infty)$, \eqref{3.48} implies equivalence of the norms $\|\, \cdot \,\|_{m,\alpha}$ and  
$|||\, \cdot \,|||_{m,\alpha}$ on $C_0^{\infty}((0,\rho))$ since repeated application of \eqref{3.48} yields, 
\begin{align}
\begin{split}
|||f|||_{m,\alpha}^2 \leq \|f\|_{m,\alpha}^2 = |||f|||_{m,\alpha}^2 + \sum_{k=0}^{m-1} |||f|||_{k,\alpha}^2 
\leq C |||f|||_{m,\alpha}^2,& \\ 
f \in C_0^{\infty}((0,\rho)), \; m \in \bbN,&
\end{split} 
\end{align} 
with $C = C(m,\alpha,\rho) \in (0,\infty)$. In particular, 
\begin{equation}
H_0^m\big((0,\rho); x^{\alpha} dx\big) = \dot {H_0^{m}}\big((0,\rho); x^{\alpha} dx\big), \quad \rho \in (0,\infty), 
\; m \in \bbN, \; \alpha \in \bbR. 
\end{equation}
Of course, since $x^{\alpha}$ is bounded from above and from below near $x = \rho$, 
\begin{align}
\begin{split}
&f \in \dot {H_0^{m}}\big((0,\rho); x^{\alpha} dx\big) = H_0^m\big((0,\rho); x^{\alpha} dx\big), \quad \rho \in (0,\infty),  \\
& \quad \text{implies } \, f(\rho) = f'(\rho) = \cdots = f^{(m-1)}(\rho) = 0.    \lb{3.53} 
\end{split} 
\end{align}

Given these preparations, we can now extend Theorem \ref{t3.1} as follows:

\begin{theorem} \lb{t3.5}
Under the hypotheses in Theorem \ref{t3.1}, items $(i)$ and $(ii)$ extend from $f \in C_0^{\infty} ((\rho, \infty))$ 
to $f \in \dot {H_0^{m}}\big((\rho,\infty); x^{\alpha} dx\big)$ and items $(iii)$ and $(iv)$ extend from 
$f \in C_0^{\infty} ((0, \rho))$  to $f \in \dot {H_0^{m}}\big((0,\rho); x^{\alpha} dx\big) 
= H_0^m\big((0,\rho); x^{\alpha} dx\big)$. 
\end{theorem}
\begin{proof}
Since the proofs of items $(i)$--$(iv)$ follow the same route based on combining Theorem \ref{t3.1} with Fatou's lemma, it suffices to focus on cases $(i)$ and $(iii)$. \\[1mm] 
$(i)$. We start with the finite interval case $(iii)$. Since $C_0^{\infty} ((0, \rho))$ is dense in  $H_0^m\big((0,\rho); x^{\alpha} dx\big)$ (in the norm $\|\, \cdot \,\|_{m,\alpha}$), given 
$f \in H_0^m\big((0,\rho); x^{\alpha} dx\big)$, there exists a sequence 
$\{f_n\}_{n \in \bbN} \subset C_0^{\infty} ((0, \rho))$ such that 
$\lim_{n \to \infty} \big\|f_n - f\big\|_{m,\alpha}^2 = 0$, explicitly,
\begin{equation}
\lim_{n \to \infty} \int_0^{\rho} dx \, x^{\alpha} \big|f_n^{(k)}(x) - f^{(k)}(x)\big|^2 = 0, \quad 0 \leq k \leq m, 
\; \alpha \in \bbR.   \lb{3.54} 
\end{equation}
Hence, for each $0 \leq k \leq m$, one can find a subsequence $\{f_{n_p,k}\}_{p \in \bbN}$ of $\{f_n\}_{n \in \bbN}$ such that 
\begin{equation}
x^{\alpha/2} f_{n_p,k}^{(k)} \underset{p \to \infty}{\longrightarrow} x^{\alpha/2} f^{(k)} \, 
\text{ pointwise a.e.~on $(0,\rho)$,}
\end{equation} 
equivalently,
\begin{equation}
f_{n_p,k}^{(k)} \underset{p \to \infty}{\longrightarrow} f^{(k)} \, 
\text{ pointwise a.e.~on $(0,\rho)$.}
\end{equation} 
Hence, abbreviating
\begin{align}
w_{\ell, \alpha, N}(x) &= A(\ell,\alpha) + B(\ell,\alpha) \sum_{k=1}^{\kk} \prod_{p=1}^{k} [\ln_{p}(\gamma/x)]^{-2} 
\no \\
& \quad + \sum_{j=2}^{\ell} |c_{2j}(\ell,\alpha)| A(j,0) [\ln(\gamma/x)]^{-2j}   \\
& \quad + \sum_{j=2}^{\ell} |c_{2j}(\ell,\alpha)|  B(j,0) \sum_{k=1}^{\kk - 1}  [\ln(\gamma/x)]^{-2j} 
\prod_{p=1}^{k} [\ln_{p+1}(\gamma/x)]^{-2}, \quad x \in (0,\rho),     \no 
\end{align}
(a well-known consequence of) Fatou's lemma (cf., e.g., \cite[Corollary~2.19]{Fo99}) and inequality \eqref{3.35} 
imply 
\begin{align}
& \int_0^{\rho} dx \, x^{\alpha - 2 \ell} w_{\ell, \alpha, N}(x) \big|f^{(m-\ell)}(x)\big|^2  \no \\
& \quad \leq \liminf_{p \to \infty} \int_0^{\rho} dx \, x^{\alpha - 2 \ell} w_{\ell, \alpha, N}(x) 
\big|f_{n_p,m-\ell}^{(m-\ell)}(x)\big|^2 \quad \text{(by Fatou's lemma)}  \no \\
& \quad = \lim_{p \to \infty} \int_0^{\rho} dx \, x^{\alpha - 2 \ell} w_{\ell, \alpha, N}(x) 
\big|f_{n_p,m-\ell}^{(m-\ell)}(x)\big|^2  \no \\
& \quad \leq \lim_{p \to \infty} \int_0^{\rho} dx \, x^{\alpha} \big|f_{n_p,m-\ell}^{(m)}(x)\big|^2 
\quad \text{(by \eqref{3.35})}   \no \\
& \quad = \int_0^{\rho} dx \, x^{\alpha} \big|f^{(m)}(x)\big|^2 \quad \text{(by \eqref{3.54} with $k=m$).}   \lb{3.57} 
\end{align} 
$(ii)$. To treat the interval $(\rho,\infty)$ one can argue as follows. Using arguments analogous to those in the proof of \cite[Proposition~3.1]{GLMW18}, one shows that the space 
\begin{align}
H_{m,\alpha}([\rho,\infty)) &= \big\{f: [\rho,\infty) \to \bbC \, \big| \, \text{for all $R > \rho$,} \, 
f^{(k)} \in AC([\rho,R]), \, 0 \leq k \leq m-1;    \no \\
& \hspace*{7mm}
f^{(k)}(\rho) = 0, \, 0 \leq k \leq m-1; \, f^{(m)} \in L^2\big((\rho,\infty); x^{\alpha}dx\big)\big\}, 
\end{align}
is a Hilbert space space with respect to the norm $|||\, \cdot \,|||_{m,\alpha}$ associated with the inner product 
\begin{equation}
\langle f,g\rangle_{m,\alpha} = \int_{\rho}^{\infty} x^{\alpha} dx \, \ol{f^{(m)}(x)} g^{(m)}(x), 
\quad f, g \in H_{m,\alpha}([\rho,\infty)).    \lb{3.58}
\end{equation}
The fact $C_0^{\infty}((\rho,\infty)) \subset H_{m,\alpha}([\rho,\infty))$ naturally leads to the introduction of the space $\dot {H_0^{m}} \big((\rho,\infty); x^{\alpha} dx\big)$ as the closure of $C_0^{\infty}((\rho,\infty))$ in 
the norm $|||\, \cdot \,|||_{m,\alpha}$ in accordance with \eqref{3.43}. Then a routine argument (see 
\cite[Appendix~B]{GLMP20b} for details) shows that if 
$f \in \dot {H_0^{m}} \big((\rho,\infty); x^{\alpha} dx\big)$ then there exists a sequence $\{f_n\}_{n \in \bbN} 
\subset C_0^{\infty}((\rho,\infty))$ such that for $0 \leq k \leq m$,
\begin{equation}
\lim_{n \to \infty} f_n^{(k)} (x) = f^{(k)} (x) \, \text{ for a.e.~$x \geq \rho$.}
\end{equation} 
At this point one can follow the Fatou-type argument in \eqref{3.57}. 
\end{proof}

\section{The Vector-Valued Case} \lb{s4}

In our final section, we establish that all previous inequalities extend line by line to the vector-valued case in 
which $f$ is $\cH$-valued, with $\cH$ a separable, complex Hilbert space. The relevance of such a generalization is briefly mentioned at the end of this section. 
 
We start by stating a power-weighted extension of \eqref{1.1} for vector-valued functions, which is derived from the more general Hardy result \cite[Example 1]{CGLMMP19} by simple iteration (see also \cite[Theorem 8.1]{GLMW18} for the special case $\alpha = 0$, $a = 0$, $b = \infty$). Inequality \eqref{5.1} will replace \eqref{1.1} in the base step of each induction proof.

\begin{lemma} \lb{l5.1}
Let $m  \in \bbN$, $\alpha \in \bbR \backslash \{2j-1\}_{1 \leq j \leq m}$, $0 \leq a < b \leq \infty$.
Then for all $f \in C_{0}^{\infty}((a,b);\cH)$,
\begin{align}\lb{5.1}
\int_{a}^{b} dx \, x^{\alpha} \big\| f^{(m )}(x) \big \|_{\cH}^{2} \geq A(m , \alpha) \int_{a}^{b} dx  \, x^{\alpha - 2m } \| f(x) \|_{\cH}^{2}.
\end{align}
The constant $A(m ,\alpha)$ is sharp and equality holds if and only if $f = 0$ on $(a,b)$.
\end{lemma}

In addition, the combined Hartman--M\"ueller-Pfeiffer transformation extends to the $\cH$-valued context. 
Indeed, given $m , \kk \in \bbN$, $\alpha \in \bbR$, $\alpha \ne 1, \dots, 2m -1$, and 
$f \in C_{0}^{\infty}((e_{\kk}, \infty); \cH)$, one sets
\begin{align}\lb{5.2}
\begin{split}
&x = e^{t}, \quad  dx = e^{t}dt, \quad  t \in (e_{\kk-1}, \infty), \\
&f(x) \equiv f(e^t) = e^{(m - \frac{1 + \alpha}{2})t}w(t), \quad w \in C_{0}^{\infty}((e_{\kk-1}, \infty); \cH),
\end{split}
\end{align}
so that 
\begin{align}\lb{5.3}
\big(x^{\alpha} f^{(m )}(x) \big)^{(m )} = e^{-(m + \frac{1 - \alpha}{2})t} \sum_{\ell = 0}^{2m } c_{\ell}(m ,\alpha) w^{(\ell)}(t).
\end{align}
Combining \eqref{5.2} and \eqref{5.3} yields
\begin{align}\lb{5.4b}
(-1)^{m } \! \Big( \! \big(x^{\alpha} f^{(m )}(x) \big)^{(m )}\!\!\!, f(x) \! \Big)_{\!\cH} \!\!\!\! = e^{-t} \! \sum_{j = 0}^{m } (-1)^{2m  - j}|c_{2j}(m ,\alpha)| \! \big( w^{(2j)}(t), w(t) \big)_{\!\cH}.
\end{align}
 Furthermore,
 \begin{align} 
&x^{\alpha - 2m } \|f(x)\|_{\cH}^{2} = e^{-t} \|w(t)\|_{\cH}^{2},    \no \\
 &x^{\alpha - 2m }\! \sum_{k=1}^{\kk}\prod_{p=1}^{k}[\ln_{p}(x)]^{-2}\|f(x)\|_{\cH}^{2}    \lb{5.4c} \\
 & \quad = e^{-t} \! \bigg\{ \! t^{-2}\|w(t)\|_{\cH}^{2}  \!+  t^{-2} \! \! \sum_{k=1}^{\kk - 1} 
 \prod_{p=1}^{k}[\ln_{p}(t)]^{-2}\|w(t)\|_{\cH}^{2} \! \bigg\}, \no 
 \end{align}
 and for $j = 2, \dots, m $,
 \begin{align} 
 &x^{\alpha - 2m }[\ln(x)]^{-2j}\|f(x)\|_{\cH}^{2} = e^{-t} t^{-2j} \|w(t)\|_{\cH}^{2},    \no \\
 &x^{\alpha - 2m } [\ln(x)]^{-2j} \sum_{k=1}^{\kk-1}  \prod_{p=1}^{k}[\ln_{p+1}(x)]^{-2}\|f(x)\|_{\cH}^{2}   
 \lb{5.4d} \\
 & \quad = e^{-t} t^{-2j}\sum_{k=1}^{\kk-1} \prod_{p=1}^{k}[\ln_{p}(t)]^{-2}\|w(t)\|_{\cH}^{2}. \no
 \end{align} 
 The modified variable transformations \eqref{3.7}, \eqref{3.21}, \eqref{3.28}, generalize analogously.

Finally, we note that \eqref{2.14} extends to the vector-valued situation in the form
\begin{equation}\lb{5.5}
\int_{a}^{b} dx \, x^{\alpha}  \big\| f^{(m )}(x) \big \|_{\cH}^{2} = (-1)^{m } \int_{a}^{b} dx \, \Big( \big(x^{\alpha} f^{(m )}(x) \big)^{(m )}, f(x) \Big)_{\cH},
\end{equation}
for $f \in C_{0}^{\infty}((a,b); \cH)$, where $0 \leq a < b \leq \infty, m  \in \bbN, \alpha \in \bbR$.

Given these preliminaries, the vector-valued case becomes completely analogous to the scalar situation treated in Section \ref{s3}:

\begin{theorem}\lb{t5.2}
Let $\ell, m , \kk \in \bbN, \alpha \in \bbR,$ and $\rho, \gamma, \tau \in (0,\infty)$. The following hold:\\[2mm]
$(i)$ If $\rho \geq e_{\kk} \gamma$ and $1 \leq \ell \leq m $, then for all $f \in C_{0}^{\infty}((\rho, \infty); \cH)$,
\begin{align}\lb{5.6}
&\int_{\rho}^{\infty} dx \, x^{\alpha} \big\| f^{(m )}(x) \big\|_{\cH}^{2}
\geq A(\ell, \alpha) \int_{\rho}^{\infty} dx \,  x^{\alpha - 2\ell}   \big\|f^{(m  - \ell)}(x) \big\|_{\cH}^{2}    \no \\
&\quad+ B(\ell,\alpha) \sum_{k=1}^{\kk} \int_{\rho}^{\infty}  dx \, x^{\alpha - 2\ell} \prod_{p=1}^{k} [\ln_{p}(x/\gamma)]^{-2}  \big\|f^{(m  - \ell)}(x) \big\|_{\cH}^{2}    \\
&\quad+ \sum_{j=2}^{\ell} |c_{2j}(\ell,\alpha)| A(j,0) \int_{\rho}^{\infty}  dx \, x^{\alpha - 2\ell} [\ln(x/\gamma)]^{-2j}   \big\|f^{(m  - \ell)}(x) \big\|_{\cH}^{2}    \no  \\
&\quad+ \sum_{j=2}^{\ell} |c_{2j}(\ell,\alpha)|  B(j,0) \sum_{k=1}^{\kk - 1} \int_{\rho}^{\infty} dx \, 
x^{\alpha - 2\ell} [\ln(x/\gamma)]^{-2j} \no \\
&\hspace*{5.6cm} \times \prod_{p=1}^{k} [\ln_{p+1}(x/\gamma)]^{-2}   \big\|f^{(m  - \ell)}(x) \big\|_{\cH}^{2}. \no 
\end{align}
$(ii)$ If $\rho \geq \tau$ and $1 \leq \ell \leq m $, then for all $f \in C_{0}^{\infty}((\rho, \infty); \cH)$,
\begin{align}\lb{5.7}
&\int_{\rho}^{\infty} dx \, x^{\alpha} \big\| f^{(m )}(x) \big\|_{\cH}^{2}
\geq A(\ell, \alpha) \int_{\rho}^{\infty} dx \,  x^{\alpha - 2\ell}   \big\|f^{(m  - \ell)}(x) \big\|_{\cH}^{2}  \no \\
&\quad+ B(\ell,\alpha) \sum_{k=1}^{\kk} \int_{\rho}^{\infty} dx \, x^{\alpha - 2\ell} \prod_{p=1}^{k} L_{p}^{2}(\tau/x)   \big\|f^{(m  - \ell)}(x) \big\|_{\cH}^{2}   \\
&\quad+ \sum_{j=2}^{\ell} |c_{2j}(\ell,\alpha)| A(j,0) \int_{\rho}^{\infty}  dx \, x^{\alpha - 2\ell} L_{1}^{2j}(\tau/x)    \big\|f^{(m  - \ell)}(x) \big\|_{\cH}^{2}  \no  \\
&\quad+ \sum_{j=2}^{\ell} |c_{2j}(\ell,\alpha)|  B(j,0) \sum_{k=1}^{\kk -1} \int_{\rho}^{\infty}  dx \, x^{\alpha - 2\ell} L_{1}^{2j}(\tau/x)  \prod_{p=1}^{k}L_{p+1}^{2}(\tau/x)  \big\|f^{(m  - \ell)}(x) \big\|_{\cH}^{2}.  \no 
\end{align}
$(iii)$ If $\gamma \geq e_{\kk} \rho$ and $1 \leq \ell \leq m $, then for all $f \in C_{0}^{\infty}((0, \rho); \cH)$,
\begin{align}\lb{5.8}
&\int_{0}^{\rho} dx \, x^{\alpha} \big\| f^{(m )}(x) \big\|_{\cH}^{2}
\geq A(\ell, \alpha) \int_{0}^{\rho} dx \,  x^{\alpha - 2\ell}    \big\|f^{(m  - \ell)}(x) \big\|_{\cH}^{2}  \no \\
&\quad+ B(\ell,\alpha) \sum_{k=1}^{\kk} \int_{0}^{\rho}  dx \, x^{\alpha - 2\ell}  \prod_{p=1}^{k} [\ln_{p}(\gamma/x)]^{-2}    \big\|f^{(m  - \ell)}(x) \big\|_{\cH}^{2}  \\
&\quad+ \sum_{j=2}^{\ell} |c_{2j}(\ell,\alpha)| A(j,0) \int_{0}^{\rho}  dx \, x^{\alpha - 2\ell}  [\ln(\gamma/x)]^{-2j}   \big\|f^{(m  - \ell)}(x) \big\|_{\cH}^{2}  \no  \\
&\quad+ \sum_{j=2}^{\ell} |c_{2j}(\ell,\alpha)|  B(j,0) \sum_{k=1}^{\kk - 1}\int_{0}^{\rho} dx \, 
x^{\alpha - 2\ell} [\ln(\gamma/x)]^{-2j}     \no \\
&\hspace*{5.5cm} \times \prod_{p=1}^{k} [\ln_{p+1}(\gamma/x)]^{-2}    \big\|f^{(m  - \ell)}(x) \big\|_{\cH}^{2}. \no 
\end{align}
$(iv)$ If $\tau \geq \rho$ and $1 \leq \ell \leq m $, then for all $f \in  C_{0}^{\infty}((0, \rho); \cH)$,
\begin{align}\lb{5.9}
&\int_{0}^{\rho} dx \, x^{\alpha} \big\| f^{(m )}(x) \big\|_{\cH}^{2}
\geq A(\ell, \alpha)\int_{0}^{\rho}  dx \,  x^{\alpha - 2\ell}  \big\|f^{(m  - \ell)}(x) \big\|_{\cH}^{2}  \no \\
&\quad+ B(\ell,\alpha) \sum_{k=1}^{\kk} \int_{0}^{\rho} dx \, x^{\alpha - 2\ell} \prod_{p=1}^{k} L_{p}^{2}(x/\tau)  \big\|f^{(m  - \ell)}(x) \big\|_{\cH}^{2}  \\
&\quad+ \sum_{j=2}^{\ell} |c_{2j}(\ell,\alpha)| A(j,0) \int_{0}^{\rho} dx \, x^{\alpha - 2\ell} L_{1}^{2j}(x/\tau)   \big\|f^{(m  - \ell)}(x) \big\|_{\cH}^{2}  \no  \\
&\quad+ \sum_{j=2}^{\ell} |c_{2j}(\ell,\alpha)|  B(j,0) \sum_{k=1}^{\kk - 1}\int_{0}^{\rho} dx \, x^{\alpha - 2\ell} L_{1}^{2j}(x/\tau)  \prod_{p=1}^{k}L_{p+1}^{2}(x/\tau)   \big\|f^{(m  - \ell)}(x) \big\|_{\cH}^{2}.  \no 
\end{align}
$(v)$ Inequalities \eqref{5.6}--\eqref{5.9} are strict for $f \not \equiv 0$ on $(\rho,\infty)$, respectively, 
$(0,\rho)$. \\[1mm]
$(vi)$ In the exceptional cases $\alpha \in \{2\ell-1\}_{1 \leq \ell \leq m}$ $($i.e., if and only if 
$A(\ell,\alpha) = 0$$)$, the first terms containing $A(\ell,\alpha)$ on the right-hand sides of 
\eqref{5.6}--\eqref{5.9} are to be deleted.
\end{theorem}

\begin{corollary}\lb{c5.3}
Let $\ell, m  \in \bbN, \alpha \in \bbR,$ and $\rho, \tau \in (0,\infty)$. Then \eqref{5.7} and \eqref{5.9} 
extend to $N = \infty$. 
\end{corollary}

Using Lemma \ref{l5.1} and identity \eqref{5.5} for the base step in the induction proof over $\kk \in \bbN$, one can follow the special scalar case treated in the proof of Theorem \ref{t3.1}, and Corollary \ref{c3.2} line by line.

As in the scalar case, the constants $A(m ,\alpha)$, $\alpha \in \bbR \backslash \{2j-1\}_{1 \leq j \leq m}$, are sharp and the inequalities extend to the associated weighted Sobolev spaces of $\cH$-valued functions; we omit the details.

We conclude with the observation that the vector-valued Hardy case (i.e., $m=1$) without logarithmic refinements (i.e., $N = 0$), played an important role in the spectral theory of $n$-dimensional Schr\"odinger operators ($n \in \bbN$, $n \geq 2$) as detailed, for instance in \cite[Chs.~IV, V]{Ku78}. In this context one employs polar coordinates and $\cH$ is then naturally identified with $L^2(S^{n-1}; d^{n-1}\omega)$. This aspect will also play a crucial role in the multi-dimensional generalizations of the results presented in this note, see \cite{GLMP20}. 

\appendix
\section{Optimality of $A(m,\alpha)$} \lb{sA}
\renewcommand{\theequation}{A.\arabic{equation}}
\renewcommand{\thetheorem}{A.\arabic{theorem}}
\setcounter{theorem}{0} \setcounter{equation}{0}

In this appendix we demonstrate sharpness of the constants $A(\ell,\alpha)$, $1 \leq \ell \leq m$. 

\begin{theorem} \lb{tA.1} 
The constants $A(\ell,\alpha)$, $1 \leq \ell \leq m$, $\alpha \in \bbR \backslash \{2j-1\}_{1 \leq j \leq \ell}$, in 
Theorems \ref{t3.1} and \ref{t3.5} are sharp. 
\end{theorem}
\begin{proof}
For simplicity, we consider the interval $(0,\rho)$ (the case $(\rho,\infty)$ being completely analogous). 

To simplify notation we assume, without loss of generality, that $\rho > 2$ for the remainder of this proof. 

We first present the proof for the case $\ell = m$ and near the end indicate the necessary changes to treat the analogous cases $1 \leq \ell \leq m-1$, $m \geq 2$. Introducing  
\begin{equation}
y_0(x) = x^{(2\ell - 1 - \alpha)/2}, \quad x > 0, \; \ell \in \bbN, \; \alpha \in \bbR,    \lb{A.1}
\end{equation}
one notes the facts
\begin{align}
&y_0^{(\ell)}(x) = 2^{- \ell} (2\ell-1-\alpha)(2\ell-3-\alpha) \cdots (3-\alpha)(1-\alpha) x^{-(1+\alpha)/2},   \\
&x^{\alpha} \big[y_0^{(\ell)}\big]^2 = A(\ell,\alpha) x^{\alpha-2\ell} [y_0(x)]^2 = A(\ell,\alpha) x^{-1},   \\
& (- 1)^{\ell} \big(x^{\alpha} y_0^{(\ell)}(x)\big)^{(\ell)} - A(\ell,\alpha) x^{\alpha - 2 \ell} y_0(x) = 0.    \lb{A.4} 
\end{align}
Next, we also introduce the cutoff functions 
\begin{align}
& \phi \in C^{\infty}(\bbR), \quad 0 \leq \phi (x) \leq 1, \, x \in \bbR, \quad \phi(x) = \begin{cases} 0, & x \leq 1, \\
1, & x \geq 2, \end{cases}    \\[1mm] 
& \phi_{\varepsilon}(x) = \phi(x/\varepsilon), \, x \in \bbR,  \, \text{ $0 < \varepsilon$ sufficiently small},   \\
&  \psi \in C^{\infty}(\bbR), \quad 0 \leq \psi (x) \leq 1, \, x \in \bbR, \quad \psi(x) = \begin{cases} 1, & x \leq \rho-2, \\
0, & x \geq \rho-1, \end{cases} 
\end{align}
and mollify $y_0$ as follows,
\begin{equation}
y_{0,\varepsilon}(x) = y_0(x) \phi_{\varepsilon}(x) \psi(x), \; 0 \leq x \leq \rho, \quad 
y_{0,\varepsilon} \in C_0^{\infty}((0,\rho)).    \lb{A.8} 
\end{equation}
Then one verifies
\begin{align}
& A(\ell,\alpha) \int_0^{\rho} dx \, x^{\alpha - 2 \ell} [y_{0,\varepsilon}(x)]^2 
= A(\ell,\alpha) \int_0^{\rho} dx \, x^{-1} \phi(x/\varepsilon)^2 \psi(x)^2    \no \\
& \quad =  A(\ell,\alpha) \int_{\varepsilon}^{\rho-2} dx \, x^{-1} \phi(x/\varepsilon)^2 
+  A(\ell,\alpha) \int_{\rho-2}^{\rho-1} dx \, x^{-1} \psi(x)^2   \no \\
& \quad =  A(\ell,\alpha) \int_{2}^{(\rho- 2)/\varepsilon} d\xi \, \xi^{-1} \phi(\xi)^2 
+ A(\ell,\alpha) \int_1^2 d\xi \, \xi^{-1} \phi(\xi)^2     \no \\
& \qquad +  A(\ell,\alpha) \int_{\rho-2}^{\rho-1} dx \, x^{-1} \psi(x)^2   \no \\
& \quad \underset{\varepsilon \downarrow 0}{=} A(\ell,\alpha) \ln(1/\varepsilon) + \Oh(1),    \lb{A.9} 
\end{align}
and 
\begin{align}
& \int_0^{\rho} dx \, x^{\alpha} \big[y_{0,\varepsilon}^{(\ell)}(x)\big]^2 
= \int_{\varepsilon}^{\rho-2} dx \, x^{\alpha} \big[y_{0,\varepsilon}^{(\ell)}(x)\big]^2  
+ \int_{\rho-2}^{\rho-1} dx \, x^{\alpha} \big[y_{0,\varepsilon}^{(\ell)}(x)\big]^2    \no \\
& \quad = \int_{\varepsilon}^{\rho-2} dx \, x^{\alpha} \big\{[(y_0(x) \phi(x/\varepsilon)]^{(\ell)}\big\}^2 
+ \int_{\rho-2}^{\rho-1} dx \, x^{\alpha} \big\{[y_0(x) \psi(x)]^{(\ell)}\big\}^2.     \lb{A.10}
\end{align}
Next, one employs
\begin{align}
& [y_0(x) \phi(x/\varepsilon)]^{(\ell)} 
= \sum_{k=0}^{\ell} \begin{pmatrix} \ell \\ k \end{pmatrix} y_0^{(\ell-k)}(x) \f{d^k}{dx^k} \phi(x/\varepsilon)   \no \\
& \quad = x^{-(1 + \alpha)/2} \sum_{k=0}^{\ell} c_{\ell,k,\alpha} (x/\varepsilon)^k \phi^{(k)}(x/\varepsilon), 
\lb{A.11} 
\end{align}
where 
\begin{align}
\begin{split}
c_{\ell,0,\alpha} &= 2^{- \ell} (2\ell -1-\alpha)(2\ell-3-\alpha) \cdots (3-\alpha)(1-\alpha), \\
c_{\ell,0,\alpha}^2 &= A(\ell,\alpha).
\end{split}
\end{align}
Thus, one can continue \eqref{A.10} as follows:
\begin{align}
\eqref{A.10} &= \int_{\varepsilon}^{\rho-2} dx \, x^{-1} \Bigg[\sum_{k=0}^{\ell} c_{\ell,k,\alpha} (x/\varepsilon)^k 
\phi^{(k)}(x/\varepsilon)\Bigg]^2     \no \\
& \quad + \int_{\rho-2}^{\rho-1} dx \, x^{\alpha} \big\{[y_0(x) \psi(x)]^{(\ell)}\big\}^2  \no \\
&=  \int_{1}^{(\rho-2)/\varepsilon} d\xi \, \xi^{-1} \Bigg[\sum_{k=0}^{\ell} c_{\ell,k,\alpha} \xi^k 
\phi^{(k)}(\xi)\Bigg]^2 + \int_{\rho-2}^{\rho-1} dx \, x^{\alpha} \big\{[y_0(x) \psi(x)]^{(\ell)}\big\}^2  \no \\
&=  \int_{1}^{(\rho-2)/\varepsilon} d\xi \, \xi^{-1} 
\Bigg[c_{\ell,0,\alpha} \phi(\xi) + \sum_{k=1}^{\ell} c_{\ell,k,\alpha} \xi^k 
\phi^{(k)}(\xi)\Bigg]^2      \no \\
& \quad + \int_{\rho-2}^{\rho-1} dx \, x^{\alpha} \big\{[y_0(x) \psi(x)]^{(\ell)}\big\}^2  \no \\
&= A(\ell,\alpha)  \int_{1}^{(\rho-2)/\varepsilon} d\xi \, \xi^{-1} \phi(\xi)^2    \no \\
& \quad + \int_{1}^{2} d\xi \, \xi^{-1} 
\Bigg\{2 c_{\ell,0,\alpha} \phi(\xi) \sum_{k=1}^{\ell} c_{\ell,k,\alpha} \xi^k 
\phi^{(k)}(\xi)     \no \\ 
& \hspace*{2.5cm} + \Bigg[\sum_{k=1}^{\ell} c_{\ell,k,\alpha} \xi^k 
\phi^{(k)}(\xi)\Bigg]^2  \Bigg\} + \int_{\rho-2}^{\rho-1} dx \, x^{\alpha} \big\{[y_0(x) \psi(x)]^{(\ell)}\big\}^2  \no \\
& \underset{\varepsilon \downarrow 0}{=} A(\ell,\alpha)  \int_{1}^{(\rho-2)/\varepsilon} d\xi \, \xi^{-1} \phi(\xi)^2 
+ \Oh(1)     \no \\
& \underset{\varepsilon \downarrow 0}{=} A(\ell,\alpha)  \int_{2}^{(\rho-2)/\varepsilon} d\xi \, \xi^{-1}  
+ A(\ell,\alpha)  \int_{1}^{2} d\xi \, \xi^{-1} \phi(\xi)^2 + \Oh(1)      \no \\
& \underset{\varepsilon \downarrow 0}{=} A(\ell,\alpha) \ln(1/\varepsilon) + \Oh(1),   \lb{A.13}
\end{align}
employing the fact that $\supp \big(\phi^{(k)}\big) \subseteq [1,2]$, $k \geq 1$. 
Thus,\eqref{A.9} and \eqref{A.13} yield
\begin{equation}
\frac{ \int_0^{\rho} dx \, x^{\alpha} \big[y_{0,\varepsilon}^{(\ell)}(x)\big]^2}{A(\ell,\alpha) \int_0^{\rho} 
dx \, x^{\alpha - 2 \ell} [y_{0,\varepsilon}(x)]^2 } \underset{\varepsilon \downarrow 0}{=} 
1 + \Oh(1/\ln(1/\varepsilon)),     \lb{A.14} 
\end{equation}
proving sharpness of $A(\ell,\alpha)$ for $\ell \in \bbN$ and 
$\alpha \in \bbR \backslash \{2j-1\}_{1 \leq j \leq \ell}$ on the function space $C_0^{\infty}((0,\rho))$. 

For $1 \leq \ell \leq m-1$, $m \geq 2$, one replaces $y_0$ by
\begin{align}
\begin{split} 
& f_0(x) = [\wti A(\ell,\alpha)/ \wti A(m,\alpha)] x^{(2m - \alpha - 1)/2}, \quad x >0, \; \alpha \in \bbR,  \\
& \wti A(\ell,\alpha) = 2^{-\ell} (2 \ell - 1 - \alpha)(2 \ell - 3 - \alpha) \cdots (3-\alpha) (1 - \alpha), \quad \alpha \in \bbR,
\end{split}
\end{align}
and observes the facts,
\begin{align}
& f_0^{m-\ell}(x) = x^{(2\ell - 1 - \alpha)/2},    \\
& f_0^{(m)}(x) = \wti A(\ell,\alpha) x^{- (\alpha + 1)/2},    \\
& x^{\alpha} \big[f_0^{(m)}(x)\big]^2 = A(\ell,\alpha) x^{\alpha - 2 \ell} \big[f_0^{(m-\ell)}(x)\big]^2 = A(\ell,\alpha) x^{-1},    
\end{align}
and then mollifies $f_0$ as before via
\begin{equation}
f_{0,\varepsilon}(x) = f_0(x) \phi_{\varepsilon}(x) \psi(x), \; 0 \leq x \leq \rho, \quad  
f_{0,\varepsilon} \in C_0((0,\infty)). 
\end{equation}
At this point one can follow the above proof step by step arriving at 
\begin{equation}
\frac{ \int_0^{\rho} dx \, x^{\alpha} \big[f_{0,\varepsilon}^{(m)}(x)\big]^2}{A(\ell,\alpha) \int_0^{\rho} 
dx \, x^{\alpha - 2 \ell} \big[f_{0,\varepsilon}^{(m-\ell)}(x)\big]^2 } \underset{\varepsilon \downarrow 0}{=} 
1 + \Oh(1/\ln(1/\varepsilon)),
\end{equation}
once more proving sharpness of $A(\ell,\alpha)$ for $\ell \in \bbN$ and 
$\alpha \in \bbR \backslash \{2j-1\}_{1 \leq j \leq \ell}$. 

Since Theorem \ref{t3.5} exhibits the same constant 
$A(\ell,\alpha)$, the latter is sharp also for the larger function space $H_0^m((0,\rho); x^{\alpha}dx)$. 
\end{proof}

\begin{remark} \lb{rA.2} 
$(i)$ Once more we recall that $A(\ell,\alpha) = 0$ if and only if $\alpha \in \{2j-1\}_{1 \leq j \leq \ell}$. Thus, the  inequality 
\begin{equation}
\int_0^{\rho} dx \, x^{\alpha} \big| f^{(m )}(x) \big|^{2}
\geq A(\ell, \alpha) \int_0^{\rho} dx \,  x^{\alpha - 2\ell} \big|f^{(m - \ell)}(x)\big|^{2}, 
\quad f \in C_0^{\infty}((0,\rho)),   \lb{A.21}
\end{equation}
is rendered trivial if $\alpha \in \{2j-1\}_{1 \leq j \leq \ell}$, with the right-hand side of \eqref{A.21} being zero. The same observation applies of course to the remaining three cases $(i)$, $(ii)$, and $(iv)$ in Theorem \ref{t3.1}. However, we emphasize that inequalities \eqref{3.33}--\eqref{3.36} remain valid and nontrivial with just the first terms on their right-hand sides removed. \\[1mm] 
$(ii)$ For $\alpha \in \bbR \backslash \{2j-1\}_{1 \leq j \leq \ell}$, inequality \eqref{A.21} extends to $\rho = \infty$, again with $A(\ell, \alpha)$ being the sharp constant for $f \in C_0^{\infty}((0,\infty))$. In particular, the  proof of Theorem \ref{tA.1}, suitably adapted, extends to the case $\rho = \infty$. (This observation applies of course to cases $(i)$, $(ii)$ (if $\rho = 0$), and $(iii)$, $(iv)$ (if $\rho = \infty$) in Theorem \ref{t3.1}). This is of course in accordance with the fact that $C_0^{\infty}((0,\rho))$-functions extended by zero beyond $\rho$, $\rho \in (0,\infty)$, can be viewed as a subset of $C_0^{\infty}((0,\infty))$. \hfill $\diamond$
\end{remark}

\begin{remark} \lb{rA.3} 
Regarding sharpness (optimality) of constants, we first note that the smaller the underlying function space, the larger the efforts needed to prove optimality. In particular, in connection with the proof presented in 
Theorem \ref{tA.1}, assuming $f \in C_0^{\infty}((0,\rho))$ requires mollification of $y_0$ in \eqref{A.1} near $x=0$ and $x=\rho$ and of course analogously in the case $f \in C_0^{\infty}((\rho,\infty))$. Many of the results cited in the remainder of this remark, under particular restrictions on the weight parameter $\alpha$, establish sharpness for larger classes of functions $f$ which do not automatically continue to hold in the $C_0^{\infty}((0,\rho))$-context. It is this simple observation that adds considerable complexity to sharpness proofs for the space $C_0^{\infty}((0,\rho))$. (By the same token, optimality proofs obtained for $C_0^{\infty}$ function spaces automatically hold for larger function spaces as long as the inequalities have already been established for the larger function spaces with the same constants $A(m,\alpha), B(m,\alpha)$.) This comment applies, in particular, to many papers that prove sharpness results in multi-dimensional situations 
for larger function spaces such as $C_0^{\infty}(B(0;\rho))$ or (homogeneous, weighted) Sobolev spaces rather than $C_0^{\infty}(B(0;\rho) \backslash \{0\})$, $B(0;\rho) \subseteq \bbR^n$ the open ball in $\bbR^n$, $n \geq 2$, with center at the origin $x=0$ and radius $\rho > 0$. Unless 
$C_0^{\infty}(B(0;\rho) \backslash \{0\})$ is dense in the appropriate norm (cf.\ the discussion preceding Theorem \ref{t3.5} in the one-dimensional context), one cannot {\it a priori} assume that the optimal constants $A(m, \wti \alpha)$ and $B(m, \wti \alpha)$ (with $\wti \alpha$ appropriately depending on $n$, e.g., 
$\wti \alpha = \alpha + n - 1$) remain the same for $C_0^{\infty}(B(0;\rho))$ and 
$C_0^{\infty}(B(0;\rho) \backslash \{0\})$, say. At least in principle, they could actually increase for the space $C_0^{\infty}(B(0;\rho) \backslash \{0\})$. In this context we emphasize that the multi-dimensional results then naturally lead to one-dimensional results for $C_0^{\infty}((0,\rho))$ upon specializing to radially symmetric functions in $C_0^{\infty} (B(0;\rho) \backslash \{0\})$. 

Sharpness of the constant $A(m ,0)$, $m  \in \bbN$ (i.e., in the unweighted case, $\alpha = 0$), in connection with the space $C_0^{\infty}((0,\infty))$ has been shown by Yafaev \cite{Ya99}. In fact, he also established this result for fractional $m$ (in this context we also refer to appropriate norm bounds in $L^p(\bbR^n; d^nx)$ of operators of the form $|x|^{-\beta} |-i \nabla|^{-\beta}$, $1 < p < n/\beta$, see \cite[Sect.~1.7]{BE11}, \cite[Sects.~1.7, 4.2]{BEL15, He77, KSWW75, KPS81, OTY11, Sc72, Si83}). 
Sharpness of $A(2,0)$ (i.e., in the unweighted Rellich case) was shown by Rellich \cite[p.~91--101]{Re69} in connection with the space $C_0^{\infty}((0,\infty))$; his multi-dimensional results also yield sharpness of $A(2,n-1)$ for $n \in \bbN$, $n \geq 3$, again for $C_0^{\infty}((0,\infty))$. In this context see also \cite[Corollary~6.3.5]{BEL15}. An exhaustive study of optimality of $A(2, \wti \alpha)$ (i.e., Rellich inequalities with power weights) for the space $C_0^{\infty}(\Omega \backslash \{0\})$ for cones 
$\Omega \subseteq \bbR^n$, $n \geq 2$, appeared in Caldiroli and Musina \cite{CM12}. The authors, in particular, describe situations where $A(2, \wti \alpha)$ has to be replaced by other constants and also treat the special case of radially symmetric functions in detail. Additional results for power weighted Rellich inequalities appeared in \cite{Mu14, Mu14a}; further extensions of power weighted Rellich inequalities with sharp constants on $C_0^{\infty}(\bbR^n \backslash \{0\})$ were obtained in \cite{MSS15}; for optimal power weighted Hardy, Rellich, and higher-order inequalities on homogeneous groups, see \cite{RS17, RY20}. Many of these references also discuss sharp (power weighted) Hardy inequalities, implying optimality for $A(1, \wti \alpha)$. Moreover, replacing $f(x)$ by $F(x) = \int_0^x dt \, f(t)$ \big(or $F(x) = \int_x^{\infty} dt \, f(t)$\big), optimality of the Hardy constant $A(1,0)$ for larger, $L^p$-based function spaces, can already be found in 
\cite[Sect.~9.8]{HLP88} (see also \cite[Theorem~1.2.1]{BEL15}, \cite[Ch.~3]{KMP07}, \cite[p.~5--11]{KPS17}, \cite{La26, Mu72, PS15}, in connection with $A(1, \alpha)$). 

Sharpness results for $A(m, \alpha)$ and $B(m, \alpha)$ together are much less frequently discussed in the literature, even under suitable restrictions on $m$ and $\alpha$. The results we found in the literature 
primarily follow upon specializing multi-dimensional results for function spaces such as 
$C_0^{\infty}(\Omega\backslash \{0\})$, or $C_0^{\infty}(\Omega)$, $\Omega \subseteq \bbR^n$ open, and appropriate restrictions on $m$, $\alpha$, and $n\geq 2$, for radially symmetric functions to the one-dimensional case at hand (cf.\ the previous paragraph). In this context we mention that the Hardy case $m=1$, without a weight function, is studied in \cite{ACR01, AE05, AS02, AVV10, BM97, Co10, DHA04, FT02, Gr03, II15, MOW13, ST16, Ta15} (all for $N=1$), and in \cite{AH12, DHA05, GM08} (all for $N \in \bbN$); the case with power weight functions is discussed in \cite{BFT18}, \cite{GM11}, \cite[Ch.~6]{GM13} 
(for $N \in \bbN$); see also \cite{MOW15}. The Rellich case $m=2$ with a general power weight on $C_0^{\infty}(\Omega\backslash \{0\})$ is discussed in \cite{CM12} (for $N=1$); the Rellich case $m=2$, without weight function on $C_0^{\infty}(\Omega)$, is studied in \cite{DHA04, DHA04b, DHA12} (all for $N=1$), the case $N \in \bbN$ is studied in \cite{AGS06}; the case of additional power weights is treated in \cite{GM11}, \cite[Ch.~6]{GM13}, \cite{Mo12}.   
The general case $m \in \bbN$ is discussed in \cite{AS09} (for $N=1$) and in \cite{Ba07}, \cite{GM11}, 
\cite[Ch.~6]{GM13}, \cite{TZ07} (all for $N \in \bbN$ and including power weights). Employing oscillation theory, sharpness of the unweighted Hardy case $A(1,0) = B(1,0) = 1/4$, with 
$N \in \bbN$, was proved in \cite{GU98}. 

The special results available on sharpness of $B(m,\alpha)$ are all saddled with enormous complexity, especially, for larger values of $N \in \bbN$. In fact, a careful proof for general $N$ will rival the length of this paper and hence has not been attempted here. \hfill $\diamond$
\end{remark}

\begin{remark} \lb{rA.4} 
The proof of optimality of $A(\ell, \alpha)$ in Theorem \ref{tA.1} consists of two principal steps: \\[1mm] 
$(i)$ Identify a function $y_0$ (see \eqref{A.1}) which is not  in $C^{\infty}_0((0,\rho))$, but which satisfies (see \eqref{A.4}), 
\begin{equation}
\frac{\int_0^{\rho} dx \, x^{\alpha} \big|y_0^{(m)}(x)\big|^2}{\int_0^{\rho} dx \, x^{\alpha - 2 \ell}  
\big|y_0^{(m - \ell)}(x)\big|^2} = A(\ell, \alpha).      \lb{A.22}
\end{equation}
$(ii)$ Exhibit a family $\{y_{0, \varepsilon}\}_{\varepsilon > 0} \subset C^{\infty}_0((0,\rho))$ of multiplicative mollifications of $y_0$ (see \eqref{A.8})  that approaches $y_0$ as $\varepsilon \downarrow 0$ and for which (see \eqref{A.14})
\begin{equation} 
\lim _{\varepsilon \downarrow 0} \frac{\int_0^{\rho} dx \, x^{\alpha} 
\big|y_{0, \varepsilon}^{(m)}(x)\big|^2}{\int_0^{\rho} dx \, x^{\alpha - 2 \ell}  
\big|y_{0, \varepsilon}^{(m - \ell)}(x)\big|^2} = A(\ell, \alpha).      \lb{A.23} 
\end{equation} 
Unfortunately, due to the ensuing complexity when having to apply the product rule of differentiation again and again, this approach in connection with $A(\ell,\alpha)$ cannot naturally be adapted to a proof of  optimality of $B(\ell, \alpha)$. The proof of optimality of $B(\ell, \alpha)$ we are currently working out requires substantial modification to steps $(i)$ and $(ii)$ above. We sketch the new approach in the special case $\ell = m$ in inequality \eqref{3.35}. We abbreviate 
\begin{equation} 
W_{m, \alpha, N}(x) = A(m,\alpha) + B(m,\alpha) \sum_{k=1}^{N-1} \prod_{p=1}^{k} [\ln_{p}(\gamma/x)]^{-2}. 
\lb{A.24}
\end{equation} 
Instead of identifying one explicit functon $y_0$ which satisfies \eqref{A.22}, we use a modification of the proof of \cite[Theorem 2]{Ba07} to identify a family 
$\{f_{0, \varepsilon} \colon (0,\rho) \rightarrow \mathbb{C}\}_{\varepsilon > 0}$ of functions  which are not  in $C^{\infty}_0((0,\rho))$ but for which 
\begin{equation} 
\lim _{\varepsilon \downarrow 0}\frac{\int_0^{\rho} dx \, x^{\alpha} \big| f_{0, \varepsilon}^{(m )}(x) \big|^{2} - \int_0^{\rho} dx \, x^{\alpha -2m} W_{m, \alpha, N}(x) | f_{0, \varepsilon} (x)|^2}{\int_0^{\rho} dx \, x^{\alpha - 2m} \prod_{p=1}^{N} [\ln_{p}(\gamma/x)]^{-2} | f_{0, \varepsilon} (x)|^2}= B(m, \alpha).     
\lb{A.25}
\end{equation}
Instead of a family of multiplicative mollifications as in \eqref{A.8},  for each $\varepsilon >0$ we employ a family $\{f_{0, \varepsilon, \nu}\}_{\nu >0} \subset C^{\infty}_0((0,\rho))$ of  mollifications of $f_{0, \varepsilon}$  using convolution with an approximate identity which has the properties:  
\begin{equation} 
\lim_{\nu \downarrow 0} \int_0^{\rho} dx\, x^{\alpha} \big| f_{0, \varepsilon, \nu}^{(m)}(x)\big|^2 
= \int_0^{\rho} dx\, x^{\alpha} | f_{0, \varepsilon}(x)|^2,     \lb{A.26}
\end{equation}
and for $k = 0, 1, \cdot \cdot \cdot , N$, 
\begin{align}
\begin{split}  
& \lim _{\nu \downarrow 0} \int_0^{\rho} dx \, x^{\alpha -2m} 
\prod_{p=1}^{k} [\ln_{p}(\gamma/x)]^{-2} | f_{0, \varepsilon, \nu} (x)|^2    \\ 
& \quad = \int_0^{\rho} dx \, x^{\alpha -2m} 
\prod_{p=1}^{k} [\ln_{p}(\gamma/x)]^{-2} | f_{0, \varepsilon} (x)|^2.      \lb{A.27}
\end{split} 
\end{align}
Thus, roughly speaking, one gets
\begin{equation} 
\lim_{\varepsilon, \nu \downarrow 0} \frac{\int_0^{\rho} dx \, x^{\alpha} \big| f_{0, \varepsilon, 
\nu}^{(m )}(x) \big|^{2} - \int_0^{\rho} dx \, x^{\alpha - 2m} 
W_{m, \alpha, N}(x) | f_{0, \varepsilon} (x)|^2}{\int_0^{\rho} dx \, x^{\alpha - 2m} \prod_{p=1}^{N} 
[\ln_{p}(\gamma/x)]^{-2} | f_{0, \varepsilon, \nu} (x)|^2}= B(m, \alpha). 
\lb{A.28} 
\end{equation} 
This approach is that much longer than the proof of Theorem \ref{tA.1} that we felt we had no choice but to write a separate paper \cite{GLMP20a} for the proof of optimality of $B(\ell, \alpha)$. \hfill $\diamond$
\end{remark}

\medskip

\noindent 
{\bf Acknowledgments.} We gratefully acknowledge discussions with Marius Mitrea. 

\medskip


\end{document}